\documentclass{article}
\usepackage[T1]{fontenc}

\usepackage[svgnames]{xcolor} 
\usepackage{amssymb} 
\usepackage{amsmath} 
\usepackage{graphicx}
\usepackage{hyperref}
\usepackage{enumitem}
\hypersetup{
    colorlinks = true,
    linkcolor = purple,
    citecolor = violet,
    breaklinks = false,
    urlcolor = blue,
}

\usepackage{geometry}

\usepackage{amsthm}
\usepackage{tikz}
\usepackage{tikz-cd}

\definecolor{blue}{rgb}{0.25,0.55,1}
\definecolor{green}{rgb}{0,0.75,0}
\DeclareMathOperator{\Aut}{Aut}

\theoremstyle{plain} 
\newtheorem{theorem}{Theorem}[section] 
\newtheorem{lemma}[theorem]{Lemma} 

\theoremstyle{definition} 
\newtheorem{definition}[theorem]{Definition}
\newtheorem{example}[theorem]{Example}
\newtheorem{construction}{Construction}

\theoremstyle{remark} 
\newtheorem{remark}[theorem]{Remark}

\title{Constructing Cospectral Vertices Through Orbits of Subgraphs}
\author{Onur Ege Erden\thanks{Current address: University of Waterloo, Department of Combinatorics, 200 University Avenue West, Waterloo, Canada. \texttt{oeerden@uwaterloo.ca}} \hspace{1cm} 
Fatihcan M. Atay\thanks{\texttt{f.atay@bilkent.edu.tr}}  \\[4pt]
\emph{Department of Mathematics, Bilkent University, 06800 Ankara, Turkey} 
}

\date{}

\begin{document}

\maketitle
\begin{abstract}

A constructive method is given for obtaining cospectral vertices in undirected graphs, along with an operation that preserves this construction. We prove that the construction yields cospectral vertices,  as well as strongly cospectral vertices under additional conditions.  
Furthermore, we generalize cospectral vertices to the case of the graph Laplacian 
and provide an analogous construction.

\end{abstract}

\section{Introduction}

Let $G(V,E)$ be a finite graph with vertex set $V$ and edge set $E$. 
We denote by $G \setminus v_i$ the subgraph obtained by deleting a vertex $v_i$ from $G$.
Two distinct vertices $v_i, v_j \in V$ are said to be \emph{cospectral} in $G$ if the spectra of the adjacency matrices of the graphs $G \setminus v_i$ and $G \setminus v_j$ are identical.

Cospectral vertices were introduced by Schwenk \cite{schwenk_cospectral_trees} for investigating whether the spectrum of the adjacency matrix determines the graph up to isomorphism. Using the non-symmetric cospectral vertex pair shown in Figure \ref{fig:cospectral_example}, he showed that the proportion of trees on $n$ vertices determined by the spectrum of their adjacency matrix goes to zero as $n \xrightarrow{} \infty$. 

\begin{figure}[h]
    \centering
    \begin{tikzpicture}[scale = 1.5]

\draw[thick]  (-6, -1.7) --  (-5, -1.7);
\draw[thick]  (-4, -1.7) --  (-5, -1.7);
\draw[thick]  (-4, -1.7) --  (-3, -1.7);
\draw[thick]  (-2, -1.7) --  (-3, -1.7);
\draw[thick]  (-2, -1.7) --  (-1, -1.7);
\draw[thick]  (0, -1.7) --  (-1, -1.7);
\draw[thick]  (0, -1.7) --  (1, -1.7);
\draw[thick]  (-1, -1.7) --  (-1, -0.7);

\fill (-6, -1.7) circle (2pt) node[left] {};
\fill (-5, -1.7) circle (2pt) node[left] {};
\fill (-4, -1.7) circle (2pt) node[left] {};
\fill[blue] (-3, -1.7) circle (2pt) node[above] {};
\fill (-2, -1.7) circle (2pt) node[right] {};
\fill (-1, -1.7) circle (2pt) node[right] {};
\fill[blue] (0, -1.7) circle (2pt) node[right] {};
\fill (1, -1.7) circle (2pt) node[above] {};

\fill (-1, -0.7) circle (2pt) node[right] {};

\node at (-3, -1.4) {$v_1$};
\node at (0, -1.4) {$v_2$};

    \end{tikzpicture}
\caption{A graph containing non-symmetric cospectral vertices $v_1$ and $v_2$ colored in blue. }\label{fig:cospectral_example}
    
\end{figure}
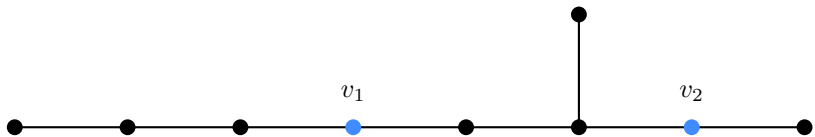

Recently, cospectral vertices have been shown to be related to phenomena in different fields beyond pure graph theory. One such example is strongly cospectral vertices, which is a stronger version of cospectrality that arises in the study of continuous quantum walks on graphs, 
where it has been observed that for perfect state transfer between two vertices to occur, these vertices have to be strongly cospectral \cite{Strongly_Cospectral,spectrally_extremal_strong_cospectrality}. 
Another example is latently symmetric vertices, introduced in \cite{Hidden_symmetry} as a generalized vertex symmetry in networks. 
Latently symmetric vertices were defined as vertices that are symmetric under a standard graph symmetry in an isospectral graph reduction of the network, which is a method that reduces the number of vertices of a graph without altering the spectrum of its adjacency matrix \cite{bunimovich_isospectral_transformations}. Having a natural sense of scale, latent symmetries were shown to be relevant for analyzing network hierarchy and decomposition of networks \cite{Latent_sym_networks}. It has been shown that cospectral vertex pairs are also latently symmetric \cite{Isospectral_red_cospectral}. Moreover, for undirected graphs the converse is also shown to be true, so cospectral vertices are identical to latently symmetric vertices \cite{Isospectral_red_cospectral}. 
Cospectral vertices have also found applications in physics where they have been used to construct lattices lacking any apparent symmetries \cite{flat_bands_latent_symmetry,topological_states_hidden_symmetry}.

Despite various studies on cospectral vertices and related concepts,
there haven't been many theoretical results that explain the existing examples of known cospectral vertices or general methods that enable the construction of new ones. 
One notable method to obtain cospectral vertices from existing ones is given in \cite{cospectrality_preserving_modifications}, where a cospectrality-preserving vertex addition/deletion operation is studied. Although not intended as a method for obtaining new cospectral vertices, a direct construction of cospectral (in fact, strongly cospectral) vertices was recently given in \cite{cayley_strongly_cospectral}  in the special case of Cayley graphs, where this construction was used to show the existence of strong cospectrality classes of arbitrarily large size in Cayley graphs. This result itself follows a similar result obtained in \cite{cayley2_strongly_cospectral}, in which the authors construct infinite families of Cayley graphs that contain a set of four pairwise strongly cospectral vertices.
Another special case focusing on strongly cospectral vertices is provided in \cite{twin_vertices_strongly_cospectral}, where the author gives conditions for preservation of strong cospectrality under Cartesian and direct product of graphs as well as join of graphs when one of the graphs is either empty or complete.

In this paper, we give a general constructive method to obtain cospectral vertices. The method, not requiring any particular graph structure, can be used to construct arbitrarily large and complex graphs with cospectral vertex pairs. We also provide an operation modifying this construction that preserves the cospectral vertices. We use this operation and the construction to explain existing examples of cospectral vertices in the literature, as well as to provide some of our own as examples. Moreover, we prove conditions that characterize when the cospectral vertices so obtained are also strongly cospectral. 
Lastly, we generalize the definition of cospectral vertices to the Laplacian matrix using a dynamical perspective. 
We note that a straightforward copy of the original definition 
does not yield identical Laplacian spectra; hence, we adopt an equivalent definition that allows the appropriate generalization.
Our definition enables us to give an analogous construction for Laplacian cospectral vertices.

\section{Notation and background}

All graphs are assumed to be nonempty, undirected, connected, and without loops. 
The notation $G(V,E)$ (often written simply as $G$) denotes a graph with finite vertex set $V$ and edge set $E$. We sometimes use the more informative forms 
$V(G)$ and $E(G)$ when it is necessary to refer to a particular graph. 
The letter $v$ in all its forms, such as $v'$, $v_i$ or $v_i^1$, always denotes a vertex of a graph. An edge between two vertices $v_i$ and $v_j$, i.e. $(v_i,v_j) \in E$, is indicated by $v_i \sim v_j$. 
    The group of graph automorphisms of $G$ that fix a vertex \( v_c \in V(G)\) is denoted by \(\mathrm{Aut}(G, v_c)\), and the orbits of vertices under this action is denoted by \( [v_j]_{v_c} \subseteq V(G)\), that is,  \( [v_j]_{v_c} \) is the set of vertices of \( G \) that are in the same orbit as \( v_j \) under the action of \(\mathrm{Aut}(G, v_c)\). 

Given a graph $G(V,E)$, we consider the vector space $\mathbb{R}^n$, where $n = |V|$. 
We usually identify $V$ with the set $\{1,2,\dots,n\}$, so a vertex $v \in V$ represents an index, and we denote the corresponding standard basis vector in $\mathbb{R}^n$ with boldface; thus, each vertex $v$ is associated with the $v$-th standard basis vector $\boldsymbol{v} \in \mathbb{R}^n$.
We denote the $v$-th component of an arbitrary vector $\boldsymbol{x} \in \mathbb{R}^n$ by $(\boldsymbol{x})_{v}$. Similarly, for a matrix $A$, we will denote the $(v_i,v_j)$-th entry as $(A)_{v_i, v_j}$ or $A_{v_i, v_j}$. 
Given a map $A: \mathbb{R}^n \xrightarrow{} \mathbb{R}^n$ and a vector $\boldsymbol{x} \in \mathbb{R}^n$, we define the subspace generated by $\boldsymbol{x}$ under the action of $A$ by $\langle \boldsymbol{x} \rangle_A = \text{span}\{\boldsymbol{x},A\boldsymbol{x},...,A^{n-1}\boldsymbol{x}\} = \text{span}\{\boldsymbol{x},A\boldsymbol{x},...,A^{n-1}\boldsymbol{x},\dots\}$. We denote the subspace orthogonal to $\langle \boldsymbol{x} \rangle_A$ by $\langle \boldsymbol{x} \rangle_A^{\perp}$.

The capital letters $A$ and $L$ denote the adjacency matrix and the Laplacian matrix of a graph, respectively, where $L=D-A$, with $D$ denoting the diagonal matrix of vertex degrees. We may also write $A(G)$ and $L(G)$ when the graph needs to be specified. When there are multiple graphs labeled in a certain way such as $\widetilde{G}$, $G'$ or $G_i$, their corresponding adjacency and Laplacian matrices will be decorated accordingly as $\widetilde{A}$ , $A'$, $A_i$, and so on.

We will denote by $\phi(G,t)$ the characteristic polynomial $ \det(tI - A)$ of the adjacency matrix $A$ of $G$ and by $\Lambda_A$ the set of eigenvalues of $A$. The following is the definition of cospectral vertices given in \cite{schwenk_cospectral_trees}.

\begin{definition}[\textit{Cospectral vertices}] \label{def:cospectral}
    Two distinct vertices $v_i, v_j \in V$ of the graph $G$ are said to be \textit{cospectral} in $G$ if the characteristic polynomial of the adjacency matrices of the graphs $G \setminus v_i$ and $G \setminus v_j$ are equal; that is, if $ \phi(G \setminus v_i, t) = \phi(G \setminus v_j,t)$.
\end{definition}

For an undirected graph $G(V,E)$, the adjacency matrix $A$ is symmetric and has the spectral decomposition $A = \sum_{\lambda \in \Lambda_A} \lambda E_{\lambda}$, where $E_{\lambda}$ are the orthogonal projection maps onto the eigenspaces of $A$. 
The following theorem gives equivalent characterizations of cospectral vertices.

\begin{theorem}\label{thm: Characterization_cospectral}\cite{Strongly_Cospectral}
    Let $G$ be a graph and let $v_i, v_j \in V(G)$ be two distinct vertices. Then, the following are equivalent:
    \begin{enumerate}
        \item $v_i$ and $v_j$ are cospectral, i.e. $\phi(G \setminus v_i,t) = \phi(G \setminus v_j,t)$.
        \item $(A^k)_{(v_i, v_i)} = (A^k)_{(v_j, v_j)}$ for all $k \in \mathbb{N}$.
        \item $(E_{\lambda})_{(v_i, v_i)} = (E_{\lambda})_{(v_j, v_j)}$ for all $\lambda \in \Lambda_A$.
        \item The subspaces $\langle \boldsymbol{v}_{i} + \boldsymbol{v}_j \rangle_{A}$ and $\langle \boldsymbol{v}_{i} - \boldsymbol{v}_j \rangle_A$ are orthogonal.
    \end{enumerate}
\end{theorem}

The third condition of Theorem \ref{thm: Characterization_cospectral} is equivalent to the projected vectors $E_{\lambda}\boldsymbol{v}_{c}$ and $E_{\lambda}\boldsymbol{v}_j$ having equal magnitude. Strong cospectrality is a stronger version of this condition:

\begin{definition}[\emph{Strongly cospectral vertices} \cite{Strongly_Cospectral}] 
Let $G(V,E)$ be a graph and let $A = \sum_{\lambda \in \Lambda_A}\lambda E_{\lambda}$ be its adjacency matrix where $E_{\lambda}$ are the orthogonal projection maps onto the eigenspaces of $A$. Two vertices $v_i,v_j \in V(G)$  are called \textit{strongly cospectral} if $E_{\lambda} \boldsymbol{v}_{i} = \pm E_{\lambda} \boldsymbol{v}_j$, $\forall\lambda \in \Lambda_A$. 
\end{definition}

We will use one of the equivalent definitions of cospectrality to generalize it to the Laplacian matrix. In order to distinguish the different notions of cospectrality, we will use the terms $A$-cospectral and $L$-cospectral, with $A$-cospectral being as in the original 
Definition \ref{def:cospectral}.
In the constructions, we indicate the relevant cospectrality relation in parentheses.

\begin{definition}[\textit{\(A\)-cospectral} and \textit{\(L\)-cospectral vertices}]

    Let \( G (V,E)\) be a graph and  let \( v_i, v_j \in V \) be two vertices of \( G \). Let $A$ and $L$ denote the adjacency matrix and the Laplacian matrix of \( G \), respectively. We say that \( v_i \) and \( v_j \) are \textit{\( A \)-cospectral}  in \( G \) if the subspaces $\langle \boldsymbol{v}_{i} + \boldsymbol{v}_j \rangle_A$ and $\langle \boldsymbol{v}_{i} - \boldsymbol{v}_j \rangle_A$  generated under the action of $A$ are orthogonal, that is, $\langle \boldsymbol{v}_{i} + \boldsymbol{v}_j \rangle_A \perp \langle \boldsymbol{v}_{i} - \boldsymbol{v}_j \rangle_A$.
  Similarly, we say that \( v_i \) and \( v_j \) are \textit{\( L \)-cospectral} in \( G \) if the subspaces \mbox{$\langle \boldsymbol{v}_{i} + \boldsymbol{v}_j \rangle_L$} and $\langle \boldsymbol{v}_{i} - \boldsymbol{v}_j \rangle_L$ generated under the action of $L$ are orthogonal, that is, $\langle \boldsymbol{v}_{i} + \boldsymbol{v}_j \rangle_L \perp \langle \boldsymbol{v}_{i} - \boldsymbol{v}_j \rangle_L$.

\end{definition}

\begin{remark}\label{rem:cospectral_definition}

We should note that there is a concept of Laplace cospectral or $L$-cospectral graphs in the literature (e.g., \cite{L_cospectral_literature, L_cospectral_literature2}), defined by having equal Laplacian spectrum. Our definition of Laplace cospectral vertices is not a direct generalization of the original cospectral vertices in this sense; since with our definition of $L$-cospectral vertices, the graphs that result after removing one of the vertices do not necessarily have equal Laplacian spectrums. We are using a different definition based on the subspaces $\langle \boldsymbol{v}_{i} + \boldsymbol{v}_j \rangle_L$ and $\langle \boldsymbol{v}_{i} - \boldsymbol{v}_j \rangle_L$ generated under the action of $L$. Although the two definitions coincide for the adjacency matrix, they are different for the Laplacian matrix. 
\end{remark}

\section{Constructing cospectral vertices}

We give an algorithmic procedure for constructing a pair of cospectral vertices using two arbitrary graphs. 

\begin{construction}[$A$-cospectral]\label{cons:A_cospectral_constructive}  
    Let \( G \) be a graph whose vertices are labeled as $v_1,v_2,\dots,v_n \in V(G)$, where $n = |V(G)|$. 
    Fix any vertex in \( V(G) \) and denote it $v_c$; we will refer to it as the \textit{fixed vertex} of the construction. 
    Let \( H \) be another graph with vertices labeled as $v'_1, v'_2, \dots, v'_r \in V(H)$, where $r = |V(H)|$. Construct the graph \( \widetilde{G} \) as follows: Take  \( H \) and two copies of \( G \), denoted as \( G^1(V^1, E^1) \) and \( G^2(V^2,E^2)\). 
    Denote the vertices of \( G^1 \) and \( G^2 \) by \( v_{1}^1, v_{2}^1, \dots, v_n^1 \) and \( v_{1}^2, v_{2}^2, \dots, v_n^2 \), respectively. Now add arbitrarily many edges between \(G^1\) and \( H \) subject to the following condition:     
For every added edge between a vertex \( v_{j}^1 \in V(G^1) \) and a vertex \mbox{\( v'_\alpha \in V(H) \)},  an edge should be added between \( v'_\alpha \) and a vertex \( v_{m}^2 \in V(G^2)\) such that \(
    v_{m} \in [v_j]_{v_c}
\); that is, \( v_j \) and \( v_{m} \) are in the same orbit in \( G \) under the action of \( \mathrm{Aut}(G, v_c) \). 
In other words, in the final graph $\widetilde{G}$ all vertices $v'_\alpha \in V(H)$ should have the same number of neighbours in corresponding orbits of the subgraphs $G^1$ and $G^2$. 
A schematic representation is shown in Figure \ref{fig:A_cospectral_constructive}. 
\end{construction}

We prove that the fixed vertex $v_c$ and its copy are a co-spectral pair.

    \begin{figure}[h]
        \centering
        \begin{tikzpicture}
            \draw (0,0) circle (2cm);
            \node at (0, 2.3) {\( G^1 \)};
            
            \begin{scope}
                \clip (0,0) circle (2cm);
                \fill[gray!20] (-2, 0.7) rectangle (2, 2);
                \fill[gray!40] (-2, -0.5) rectangle (2, 0.7);
                \fill[gray!60] (-2, -1.7) rectangle (2, -0.5);
                \fill[gray!100] (-2, -2) rectangle (2, -1.7);
            \end{scope}
            \draw (-1.9, 0.7) -- (1.9, 0.7);
            \draw (-1.95, -0.5) -- (1.95, -0.5);
            \draw (-1.05, -1.7) -- (1.05, -1.7);
            
            \draw (6,0) circle (2cm);
            \node at (6, 2.3) {\( G^2 \)};
            
            \begin{scope}
                \clip (6,0) circle (2cm);
                \fill[gray!20] (4, 0.7) rectangle (8, 2);
                \fill[gray!40] (4, -0.5) rectangle (8, 0.7);
                \fill[gray!60] (4, -1.7) rectangle (8, -0.5);
                \fill[gray!100] (4, -2) rectangle (8, -1.7);
            \end{scope}
            \draw (4.1, 0.7) -- (7.9, 0.7);
            \draw (4.05, -0.5) -- (7.95, -0.5);
            \draw (5.05, -1.7) -- (7.05, -1.7);

            \draw (3,3) circle (1.3cm);
            \node at (3, 4.6) {\( H \)};
            
            \fill (0.5, 1.2) circle (2pt) node[above] {\( v_{2}^1 \)};

            \fill (0, 0) circle (2pt) node[above] {\( v_{1}^1 \)};
            
            \fill (0, -2) circle (2pt) node[below] {\( v_{c}^1 \)};
            
            \fill (6.5, 1.2) circle (2pt) node[above] {\( v_{2}^2 \)};
            \fill (5.2, 1.2) circle (2pt) node[above] {\( v_{3}^2 \)};

            \fill (4.7, 0.3) circle (2pt) node[below] {\( v_{4}^2 \)};

            \fill (6, -2) circle (2pt) node[below] {\( v_{c}^2 \)};

            \fill (3, 2) circle (2pt) node[above] {\( v'_{1} \)};
            \fill (3.5, 2.9) circle (2pt) node[above] {\( v'_{2} \)};
            \fill (2.3, 2.5) circle (2pt) node[above] {\( v'_{3} \)};
            
            \draw (3, 2) -- (4.7, 0.3);
            \draw (3, 2) -- (0, 0);

            \draw (3.5, 2.9) -- (0.5, 1.2);
            \draw (3.5, 2.9) -- (6.5, 1.2);

            \draw (2.3, 2.5) -- (0.5, 1.2);
            \draw (2.3, 2.5) -- (5.2, 1.2);
        \end{tikzpicture}
        \caption{Construction \ref{cons:A_cospectral_constructive} illustrated. The graphs \( G^1 \) and \( G^2 \) are divided into orbits under the action of \(\mathrm{Aut}(G, v_c)\), shown by shades of gray, where the darkest gray contains only \(v_{c}\), the fixed vertex of the construction. The vertices \(v_{1}\), \(v_{4}\) and \(v_{2}\), \(v_{3}\) are in the same orbit. For every edge from a vertex in \(G^1\) to a vertex \(v'_\alpha \in V(H)\), there is a corresponding edge to \(v'_\alpha \in V(H)\) from a vertex of \(G^2\) in the same orbit. Theorem \ref{thm:A_cospectral_constructive} proves that the vertices \( v_{c}^1 \) and \( v_{c}^2 \) are $A$-cospectral in the union graph \( \widetilde{G} \).}\label{fig:A_cospectral_constructive}
      
    \end{figure}
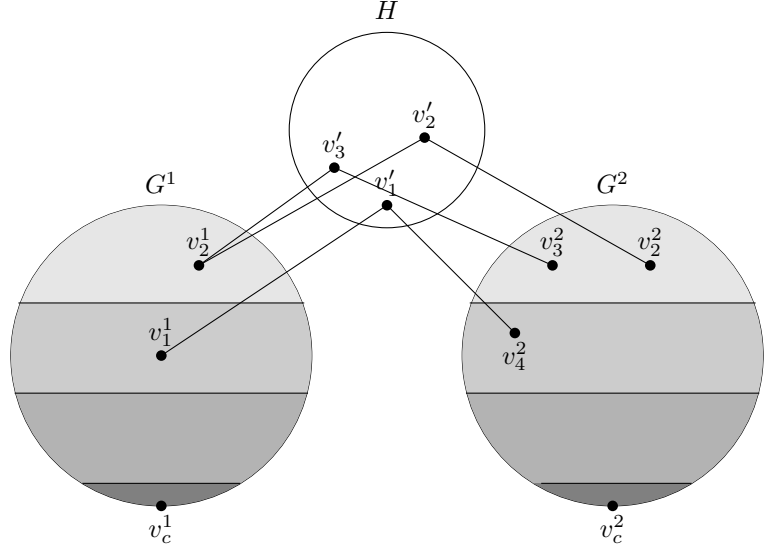

\begin{theorem}\label{thm:A_cospectral_constructive}  
Let $G$ and $H$ be two graphs and let $\widetilde{G}$ be a graph constructed as in Construction \ref{cons:A_cospectral_constructive}, with the fixed vertex $v_c \in V(G)$ and subgraphs $G^1$, $G^2$, $H$. Then, \( v_{c}^1 \) and \( v_{c}^2 \) are $A$-cospectral in \( \widetilde{G} \).     
\end{theorem}

    \begin{proof}
    
    Let $\widetilde{A}$ denote the adjacency matrix of $\widetilde{G}$. We claim that the following hold for all \( k \in \mathbb{N} \):

    \begin{enumerate}[label=(\roman*)]
        \item\label{thm1:claim1}  $(\widetilde{A}^k (\boldsymbol{v}_{c}^1 - \boldsymbol{v}_{c}^2))_{v'_\alpha} = 0$ for all $ v'_\alpha \in V(H)$ 
        \item\label{thm1:claim2}  $(\widetilde{A}^k (\boldsymbol{v}_{c}^1 - \boldsymbol{v}_{c}^2))_{v_{m}^1} = -(\widetilde{A}^k (\boldsymbol{v}_{c}^1 - \boldsymbol{v}_{c}^2))_{v_{m}^2}$ for all $ v_{m} \in V(G)$
        \item\label{thm1:claim3}  $(\widetilde{A}^k (\boldsymbol{v}_{c}^1 - \boldsymbol{v}_{c}^2))_{v_{m}^1} = (\widetilde{A}^k (\boldsymbol{v}_{c}^1 - \boldsymbol{v}_{c}^2))_{v_{\ell}^1}$ for any $v_{m}, v_\ell \in V(G)$ such that $v_\ell \in [v_m]_{v_c}$.
    \end{enumerate}
    Clearly, claims \ref{thm1:claim2} and \ref{thm1:claim3} together also imply       \[
    (\widetilde{A}^k (\boldsymbol{v}_{c}^1 - \boldsymbol{v}_{c}^2))_{v_{m}^2} = (\widetilde{A}^k (\boldsymbol{v}_{c}^1 - \boldsymbol{v}_{c}^2))_{v_{\ell}^2} \;\; \text{for any } v_{m}, v_\ell \in V(G) \text{ such that } v_\ell \in [v_m]_{v_c}.
    \]
We prove by induction on $k$.
   
    \textit{Base Case \( k = 0 \):} Trivial.  
    
    \textit{Inductive Step:} Suppose claims \ref{thm1:claim1}, \ref{thm1:claim2} and \ref{thm1:claim3} hold for $k - 1$. We will show they also hold for $k$. For any \( v'_\alpha \in V(H) \) we have  
\begin{multline}\label{eq:thm1:eq1}
        (\widetilde{A}^k (\boldsymbol{v}_{c}^1 - \boldsymbol{v}_{c}^2))_{v'_\alpha} =
    \sum_{\substack{v'_\beta \sim v'_\alpha \\ v'_\beta \in V(H)}} (\widetilde{A}^{k-1} (\boldsymbol{v}_{c}^1 - \boldsymbol{v}_{c}^2))_{v'_\beta} \\ 
    + \sum_{\substack{v_{m}^1 \sim v'_\alpha \\ v_m \in V(G)}} (\widetilde{A}^{k-1} (\boldsymbol{v}_{c}^1 - \boldsymbol{v}_{c}^2))_{v_{m}^1} + \sum_{\substack{v_{\ell}^2 \sim v'_\alpha \\ v_\ell \in V(G)}} (\widetilde{A}^{k-1} (\boldsymbol{v}_{c}^1 - \boldsymbol{v}_{c}^2))_{v_{\ell}^2}.
    \end{multline}
By the induction hypothesis for claim \ref{thm1:claim1}, we have $ (\widetilde{A}^{k-1} (\boldsymbol{v}_{c}^1 - \boldsymbol{v}_{c}^2))_{v'_\beta} = 0$ for all $v'_\beta \in V(H)$. So, $\sum_{\substack{v'_\beta \sim v'_\alpha \\ v'_\beta \in V(H)}} (\widetilde{A}^{k-1} (\boldsymbol{v}_{c}^1 - \boldsymbol{v}_{c}^2))_{v'_\beta} = 0$. 
By Construction \ref{cons:A_cospectral_constructive}, for every vertex $v_m \in V(G)$ such that $v_m^1 \sim v_\alpha'$, there exists a corresponding vertex $v_{\ell_m} \in [v_m]_{v_c}$ such that $v_{\ell_m}^2 \sim v_\alpha'$. 
Using these facts, the right-hand side of equation (\ref{eq:thm1:eq1}) becomes
 \begin{multline*}
    \sum_{\substack{v_{m}^1 \sim v'_\alpha \\ v_m \in V(G)}} (\widetilde{A}^{k-1} (\boldsymbol{v}_{c}^1 - \boldsymbol{v}_{c}^2))_{v_{m}^1} + \sum_{\substack{v_{\ell}^2 \sim v'_\alpha \\ v_\ell \in V(G)}} (\widetilde{A}^{k-1} (\boldsymbol{v}_{c}^1 - \boldsymbol{v}_{c}^2))_{v_{\ell}^2} 
    \\
    = \sum_{\substack{v_{m}^1 \sim v'_\alpha \\ v_m \in V(G)}} \big[(\widetilde{A}^{k-1} (\boldsymbol{v}_{c}^1 - \boldsymbol{v}_{c}^2))_{v_{m}^1} + (\widetilde{A}^{k-1} (\boldsymbol{v}_{c}^1 - \boldsymbol{v}_{c}^2))_{v_{\ell_m}^2}]  
 \end{multline*}
where \( v_{\ell_m} \in [v_m]_{v_c} \) by construction. Using the induction hypothesis for claims \ref{thm1:claim2} and \ref{thm1:claim3}, respectively, in the first and second equalities below, we obtain
   \begin{multline*}
    \sum_{\substack{v_{m}^1 \sim v'_\alpha \\ v_m \in V(G)}} \big[(\widetilde{A}^{k-1} (\boldsymbol{v}_{c}^1 - \boldsymbol{v}_{c}^2))_{v_{m}^1} + (\widetilde{A}^{k-1} (\boldsymbol{v}_{c}^1 - \boldsymbol{v}_{c}^2))_{v_{\ell_m}^2}]  
    \\
    = \sum_{\substack{v_{m}^1 \sim v'_\alpha \\ v_m \in V(G)}} \big[(\widetilde{A}^{k-1} (\boldsymbol{v}_{c}^1 - \boldsymbol{v}_{c}^2))_{v_{m}^1} - (\widetilde{A}^{k-1} (\boldsymbol{v}_{c}^1 - \boldsymbol{v}_{c}^2))_{v_{\ell_m}^1}] = 0.
     \end{multline*} 
Hence, the right-hand side of (\ref{eq:thm1:eq1}) is zero. This proves claim \ref{thm1:claim1} for  \( k \). To prove claim \ref{thm1:claim2}, note that for any \( v_{m}^1 \in V(G^1) \),  
    \[
    (\widetilde{A}^k (\boldsymbol{v}_{c}^1 - \boldsymbol{v}_{c}^2))_{v_{m}^1}
    = \sum_{\substack{v_{\ell} \sim v_{m} \\ v_{\ell} \in V(G)}} (\widetilde{A}^{k-1} (\boldsymbol{v}_{c}^1 - \boldsymbol{v}_{c}^2))_{v_{\ell}^1}
    + \sum_{\substack{v'_\alpha \sim v_{m}^1\\ v'_\alpha \in V(H)}} (\widetilde{A}^{k-1} (\boldsymbol{v}_{c}^1 - \boldsymbol{v}_{c}^2))_{v'_\alpha}.
    \]
By the induction hypothesis for claim \ref{thm1:claim1}, the second term is zero; therefore,
    \[
    (\widetilde{A}^k (\boldsymbol{v}_{c}^1 - \boldsymbol{v}_{c}^2))_{v_{m}^1}
    = \sum_{\substack{v_{\ell} \sim v_{m} \\ v_{\ell} \in V(G)}} (\widetilde{A}^{k-1} (\boldsymbol{v}_{c}^1 - \boldsymbol{v}_{c}^2))_{v_{\ell}^1}.
    \]
By the induction hypothesis for claims \ref{thm1:claim2} and then \ref{thm1:claim1}, we have
    \begin{align*}
    \sum_{v_\ell \sim v_{m}} (\widetilde{A}^{k-1} (\boldsymbol{v}_{c}^1 - \boldsymbol{v}_{c}^2))_{v_{\ell}^1} &= -\sum_{\substack{v_{\ell} \sim v_{m} \\ v_{\ell} \in V(G)}} (\widetilde{A}^{k-1} (\boldsymbol{v}_{c}^1 - \boldsymbol{v}_{c}^2))_{v_{\ell}^2} 
    \\
    &= -\sum_{\substack{v_{\ell} \sim v_{m} \\ v_{\ell} \in V(G)}} (\widetilde{A}^{k-1} (\boldsymbol{v}_{c}^1 - \boldsymbol{v}_{c}^2))_{v_{\ell}^2} - \sum_{\substack{v'_\alpha \sim v_{m}^2\\ v'_\alpha \in V(H)}} (\widetilde{A}^{k-1} (\boldsymbol{v}_{c}^1 - \boldsymbol{v}_{c}^2))_{v'_\alpha}
    \\
    &=  -(\widetilde{A}^k (\boldsymbol{v}_{c}^1 - \boldsymbol{v}_{c}^2))_{v_{m}^2},
    \end{align*}
which proves claim  \ref{thm1:claim2}.  Finally to prove claim  \ref{thm1:claim3}, let \( v_m, v_\ell \in V(G) \) be two vertices such that \( v_\ell \in [v_m]_{v_c}\). Then there must exist a graph automorphism $\gamma \in \mathrm{Aut}(G,v_c)$ such that $\gamma (v_\ell) = v_m$. Now, 
    \begin{align*}
    (\widetilde{A}^k (\boldsymbol{v}_{c}^1 - \boldsymbol{v}_{c}^2))_{v_{m}^1}
    &= \sum_{\substack{v_{j} \sim v_{m} \\ v_{j} \in V(G)}} (\widetilde{A}^{k-1} (\boldsymbol{v}_{c}^1 - \boldsymbol{v}_{c}^2))_{v_{j}^1}
    + \sum_{\substack{v'_\alpha \sim v_{m}^1\\ v'_\alpha \in V(H)}} (\widetilde{A}^{k-1} (\boldsymbol{v}_{c}^1 - \boldsymbol{v}_{c}^2))_{v'_\alpha}.
    \end{align*}
Since the last term is zero by the induction hypothesis for claim \ref{thm1:claim1}, 
     \begin{align}\label{eq:thm1:eq2}
      (\widetilde{A}^k (\boldsymbol{v}_{c}^1 - \boldsymbol{v}_{c}^2))_{v_{m}^1}
    &= \sum_{\substack{v_{j} \sim v_{m} \\ v_{j} \in V(G)}} (\widetilde{A}^{k-1} (\boldsymbol{v}_{c}^1 - \boldsymbol{v}_{c}^2))_{v_{j}^1}.
     \end{align}
By the fact that corresponding neighbors of \( v_{m} \) and \( v_{\ell} \) are in the same automorphism orbits of $G$ via \( \gamma \) and by the induction hypothesis for claim \ref{thm1:claim3}, we have 
\begin{align*}
\sum_{\substack{v_{j} \sim v_{m} \\ v_{j} \in V(G)}} (\widetilde{A}^{k-1} (\boldsymbol{v}_{c}^1 - \boldsymbol{v}_{c}^2))_{v_{j}^1} = \sum_{\substack{v_{s} \sim v_{\ell} \\ v_{s} \in V(G)}} (\widetilde{A}^{k-1} (\boldsymbol{v}_{c}^1 - \boldsymbol{v}_{c}^2))_{v_{s}^1}
\end{align*}
Since $(\widetilde{A}^k (\boldsymbol{v}_{c}^1 - \boldsymbol{v}_{c}^2))_{v_{m}^1}
    =  \sum_{\substack{v_{j} \sim v_{m} \\ v_{j} \in V(G)}} (\widetilde{A}^{k-1} (\boldsymbol{v}_{c}^1 - \boldsymbol{v}_{c}^2))_{v_{j}^1}$ for any vertex $v_m \in V(G)$, equation (\ref{eq:thm1:eq2}) now implies     
    \begin{align*}
    (\widetilde{A}^k (\boldsymbol{v}_{c}^1 - \boldsymbol{v}_{c}^2))_{v_{m}^1}
    &=  \sum_{\substack{v_{j} \sim v_{m} \\ v_{j} \in V(G)}} (\widetilde{A}^{k-1} (\boldsymbol{v}_{c}^1 - \boldsymbol{v}_{c}^2))_{v_{j}^1} \\
    &= \sum_{\substack{v_{s} \sim v_{\ell} \\ v_{s} \in V(G)}} (\widetilde{A}^{k-1} (\boldsymbol{v}_{c}^1 - \boldsymbol{v}_{c}^2))_{v_{s}^1}= (\widetilde{A}^k (\boldsymbol{v}_{c}^1 - \boldsymbol{v}_{c}^2))_{v_{\ell}^1}
    \end{align*}
proving claim \ref{thm1:claim3} and completing the induction argument. 

Now we have proved claim \ref{thm1:claim2}, which says that  
\[
    (\widetilde{A}^k (\boldsymbol{v}_{c}^1 - \boldsymbol{v}_{c}^2))_{v_{m}^1} = - (\widetilde{A}^k (\boldsymbol{v}_{c}^1 - \boldsymbol{v}_{c}^2))_{v_{m}^2}, \quad  \forall v_m \in V(G) \text{ and } k \in \mathbb{N}.
    \]
In particular, this implies that for all \( k \in \mathbb{N} \),  
    \[
    (\widetilde{A}^k (\boldsymbol{v}_{c}^1 - \boldsymbol{v}_{c}^2))_{v_{c}^1} + (\widetilde{A}^k (\boldsymbol{v}_{c}^1 - \boldsymbol{v}_{c}^2))_{v_{c}^2} = (\widetilde{A}^k (\boldsymbol{v}_{c}^1 - \boldsymbol{v}_{c}^2))^T (\boldsymbol{v}_{c}^1 + \boldsymbol{v}_{c}^2) = 0.
    \]Thus, we conclude that $\boldsymbol{v}_{c}^1 + \boldsymbol{v}_{c}^2 \perp \langle \boldsymbol{v}_{c}^1 - \boldsymbol{v}_{c}^2 \rangle_{\widetilde{A}}$, or (since $\widetilde{A}$ is symmetric) equivalently, 
\[
    \langle \boldsymbol{v}_{c}^1 + \boldsymbol{v}_{c}^2 \rangle_{\widetilde{A}} \perp \langle \boldsymbol{v}_{c}^1 - \boldsymbol{v}_{c}^2 \rangle_{\widetilde{A}}.
    \]
Hence, \( v_{c}^1 \) and \( v_{c}^2 \) are $A$-cospectral in \(\widetilde{G}\).  
\end{proof}

\begin{example}\label{ex:A_cospectral_constructive1}
We provide a simple example in Figure \ref{fig:A_cospectral_constructive1} to Construction  \ref{cons:A_cospectral_constructive}. In the figure, the subgraphs \(G^1\) and \(G^2\) are highlighted with red vertices and edges and the black vertices can be regarded as \(H\). The fixed vertex of this construction is $v_c = v_1$. The vertices \(v_{1}^1\) and \(v_{1}^2\) are indicated with blue. 
It can be seen that the red vertices \(v_{2}\), \(v_{3}\) and \(v_{4}\) are all in the same orbit under $\Aut(G,v_1)$.
The edges connecting vertices of \(H\) and $G^1$, $G^2$ also satisfy the requirements of Construction \ref{cons:A_cospectral_constructive}. By Theorem \ref{thm:A_cospectral_constructive}, vertices \(v_{1}^1\) and \(v_{1}^2\) are $A$-cospectral. 

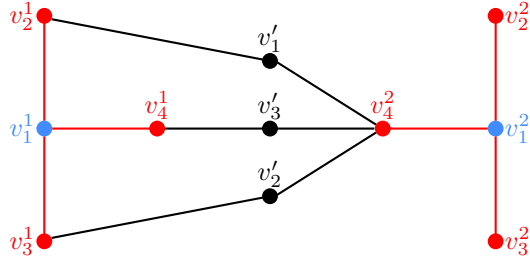
\begin{figure}[h]
    \centering
    \begin{tikzpicture}[scale = 1.5]
    
        \fill[blue] (-2, -0) circle (2pt) node[left] {\(v_{1}^1\)};
        \fill[red] (-2, 1)   circle (2pt) node[left] {\(v_{2}^1\)};
        \fill[red] (-2, -1)  circle (2pt) node[left] {\(v_{3}^1\)};
        \fill[red] (-1, -0)  circle (2pt) node[above] {\(v_{4}^1\)};

        \fill[blue] (2, -0) circle (2pt) node[right] {\(v_{1}^2\)};
        \fill[red] (2, 1)   circle (2pt) node[right] {\(v_{2}^2\)};
        \fill[red] (2, -1)  circle (2pt) node[right] {\(v_{3}^2\)};
        \fill[red] (1, -0)  circle (2pt) node[above] {\(v_{4}^2\)};

        \draw[red, thick] (-2,0.07) -- (-2,1);
        \draw[red, thick] (-2,-0.07) -- (-2,-1);
        \draw[red, thick] (-1.93,0) -- (-1,0);

        \draw[red, thick] (2,0.07) -- (2,1);
        \draw[red, thick] (2,-0.07) -- (2,-1);
        \draw[red, thick] (1.93,0) -- (1,0);

        \fill (0, 0.6)  circle (2pt) node[above] {\(v_{1}'\)};
        \fill (0, -0.6) circle (2pt) node[above] {\(v_{2}'\)};
        \fill (0, 0)    circle (2pt) node[above] {\(v_{3}'\)};

        \draw[thick] (-1.935,0.97) -- (-0, 0.6);
        \draw[thick] (-1.935,-0.97) -- (-0, -0.6);
        \draw[thick] (-0.93,0) -- (-0, 0);

        \draw[thick] (0.94,0.035) -- (0.06, 0.59);
        \draw[thick] (0.94,-0.035) -- (0.06, -0.59);
        \draw[thick] (0.925,0) -- (0, 0);
    \end{tikzpicture}
\caption{The red subgraphs together with the blue vertices are isomorphic to one another and represent the subgraphs $G^1$ and $G^2$ of Construction \ref{cons:A_cospectral_constructive}. Their connections to the black vertices in the middle, which represent the subgraph $H$, satisfy the conditions of Construction \ref{cons:A_cospectral_constructive} and so by Theorem \ref{thm:A_cospectral_constructive}, vertices $v_1^1$ and $v_1^2$ are $A$-cospectral.}\label{fig:A_cospectral_constructive1}
    
\end{figure}
\end{example}

\begin{example}\label{ex:A_cospectral_constructive2}
    Another example of cospectral vertices, taken from \cite{Isospectral_red_cospectral}, is depicted in Figure \ref{fig:A_cospectral_constructive2}, where the cospectrality follows by Construction \ref{cons:A_cospectral_constructive} and Theorem \ref{thm:A_cospectral_constructive}.

    \begin{figure}[h]
        \centering
        \begin{tikzpicture}[scale = 1.5]
            
            \fill[red] (-0.9, -0.5) circle (2pt) node[above] {};
            \fill[blue] (-0.9, 0.5) circle (2pt) node[left] {};
            \fill[red] (0, 0) circle (2pt) node[below] {};

            \fill[red] (2.1, 0.5) circle (2pt) node[above] {};
            \fill[red] (2.1, -0.5) circle (2pt) node[left] {};
            \fill[blue] (3, 0) circle (2pt) node[below] {};
            
            \fill (1, 1.4) circle (2pt) node[left]{};
            \fill (1, 0.7) circle (2pt) node[below] {};
            \fill (1, -0.7) circle (2pt) node[below] {};

            \draw[red, thick] (-0.9, 0.43) --  (-0.9, -0.5);
            \draw[red, thick] (-0.85, 0.47) --  (0, 0);
            \draw[red, thick] (-0.9, -0.5) --  (0, 0);

            \draw[red, thick] (2.95, 0.02) --  (2.1, 0.5);
            \draw[red, thick] (2.95, -0.02) --  (2.1, -0.5);
            \draw[red, thick](2.1, -0.5) --  (2.1, 0.5);
            
            \draw[thick]  (1, 1.4) -- (1, 0.7);
            \draw[thick]  (1, 0.7) -- (0.057, 0.046);
            \draw[thick]  (1, 0.7) -- (2.03, 0.52);
            \draw[thick]  (1, -0.7) -- (2.03, -0.52);
            \draw[thick]  (1, -0.7) -- (0.051, -0.051);

        \end{tikzpicture}
        \caption{The red subgraphs correspond to $G^1$ and $G^2$ in Construction \ref{cons:A_cospectral_constructive}, and the blue vertices correspond to the $A$-cospectral vertices.}\label{fig:A_cospectral_constructive2}
    \end{figure}
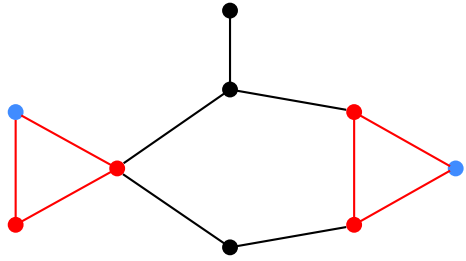

\end{example}

\begin{remark}
The vector space $\mathbb{R}^n$ associated with the graph $G$, together with the adjacency matrix $A$ of a graph, can be thought of as a dynamical system in which $\mathbb{R}^n$ is the phase space and $A$ is the linear map defining the dynamics. The vectors $\boldsymbol{x} \in \mathbb{R}^n$ then become the ``states" of the dynamical system, having values on the vertices of $G$. From this point of view, Theorem \ref{thm:A_cospectral_constructive} becomes clearer. Indeed, the theorem works because in the dynamical system of the graph $\widetilde{G}$ induced by its adjacency matrix $\widetilde{A}$, the ``initial condition" $\boldsymbol{v}_{c}^1 - \boldsymbol{v}_{c}^2$ follows the trajectory $\boldsymbol{v}_{c}^1 - \boldsymbol{v}_{c}^2, \widetilde{A}(\boldsymbol{v}_{c}^1 - \boldsymbol{v}_{c}^2), \widetilde{A}^2(\boldsymbol{v}_{c}^1 - \boldsymbol{v}_{c}^2)$ and so on, whose states $\widetilde{A}^k (\boldsymbol{v}_{c}^1 - \boldsymbol{v}_{c}^2)$ always have the value 0 in their components corresponding to the subgraph $H$; that is,
\[
(\widetilde{A}^k (\boldsymbol{v}_{c}^1 - \boldsymbol{v}_{c}^2))_{v'_\alpha} = 0,\quad \forall v'_\alpha \in V(H), \forall k \in \mathbb{N}.
\]In essence, this means that with this initial condition, the dynamical system is not affected by the presence of these vertices and in practice behaves like two identical disjoint subdynamics on the subgraphs $G^1$ and $G^2$, which is captured by 
\[ 
(\widetilde{A}^k (\boldsymbol{v}_{c}^1 - \boldsymbol{v}_{c}^2))_{v_{m}^1} = -(\widetilde{A}^k (\boldsymbol{v}_{c}^1 - \boldsymbol{v}_{c}^2))_{v_{m}^2},\quad \forall v_m \in V(G), \forall k \in \mathbb{N}.
\] 
This last equality implies $\boldsymbol{v}_{c}^1 + \boldsymbol{v}_{c}^2 \perp \langle \boldsymbol{v}_{c}^1 - \boldsymbol{v}_{c}^2 \rangle_{\widetilde{A}}$, which is equivalent to cospectrality.
    
\end{remark}

One can modify Construction \ref{cons:A_cospectral_constructive} while still preserving the cospectrality as long as one does not break the inherent dynamical symmetry. One such operation is given below in Construction \ref{cons:modify_cospectral_construction}.  

\begin{construction}[$A$-cospectral]\label{cons:modify_cospectral_construction}

Let $\widetilde{G}$ be a graph obtained by Construction \ref{cons:A_cospectral_constructive} from graphs $G$ and $H$, with the fixed vertex $v_c \in V(G)$. 
Consider an orbit $[v_j]_{v_c}$ of vertices under the action of $\Aut(G, v_c)$ on $G$. In the graph $\widetilde{G}$, there are two copies of $G$ by construction, $G^1$ and $G^2$, so, these orbits also have two copies, denoted by
\[
[v_j^1]_{v_c} \subseteq V(G^1), \quad [v_j^2]_{v_c} \subseteq V(G^2).
\]
Now construct a graph $\widehat{G}$ from $\widetilde{G}$ by connecting vertices of $[v_j^1]_{v_c}$ with vertices of $[v_j^2]_{v_c}$ in a bijective manner; namely, connect each  $ v_{m}^1 \in [v_j^1]_{v_c}$ to exactly one vertex in $[v_j^2]_{v_c}$ such that every $v_{k}^2 \in [v_j^2]_{v_c}$ is also connected to exactly one vertex in $[v_j^1]_{v_c}$.
    
\end{construction}

\begin{lemma}\label{lem:modify_cospectral_construction} 
Let $\widetilde{G}$ be a graph obtained by Construction \ref{cons:A_cospectral_constructive}, 
with the pair of $A$-cospectral vertices $v_c^1$ and $v_c^2$.
If $\widetilde{G}$ is modified as explained in Construction \ref{cons:modify_cospectral_construction} to obtain a new graph $\widehat{G}$, then, $v_c^1$ and $v_c^2$ remain $A$-cospectral in $\widehat{G}$.  
\end{lemma}

\begin{remark}\label{rem:modify_cospectral_construction}
Note that one can use Construction \ref{cons:modify_cospectral_construction} in different ways to create different graphs. An example is given in Figure \ref{fig:different_one_to_one}.
\end{remark}

    \begin{figure}[h]
        \centering
      \begin{tikzpicture}[scale = 1.5]
            
            \fill[red] (-9.5, 0.8) circle (2pt) node[above] {};
            \fill[blue] (-10, 0) circle (2pt) node[left] {};
            \fill[red] (-9.5, -0.8) circle (2pt) node[below] {};

            \fill[red] (-8.5, 0.8) circle (2pt) node[above] {};
            \fill[red] (-8.5, -0.8) circle (2pt) node[left] {};
            \fill[blue] (-8, 0) circle (2pt) node[below] {};
            
            \fill (-9, 0.4) circle (2pt) node[left]{};
            \fill (-9, 1.2) circle (2pt) node[below] {};
            \fill (-9, -1.2) circle (2pt) node[below] {};

            \draw[red, thick] (-9.968, 0.058) -- (-9.5, 0.8);
            \draw[red, thick] (-9.968, -0.058) -- (-9.5, -0.8);

            \draw[green, thick] (-8.565, 0.8) -- (-9.435, 0.8);
            \draw[red, thick] (-8.5, 0.8) -- (-8.032, 0.058);
            \draw[red, thick](-8.5, -0.8) -- (-8.032, -0.058);
            \draw[green, thick] (-8.565, -0.8) -- (-9.435, -0.8);
            
            \draw[thick]  (-9, 0.4) -- (-9.478, -0.747);
            \draw[thick]  (-9, 0.4) -- (-8.545, 0.761);
            \draw[thick]  (-9, 1.2) -- (-8.545, 0.839);
            \draw[thick]  (-9, 1.2) -- (-9.455, 0.839);
            \draw[thick]  (-9, -1.2) -- (-8.518, 0.732);
            \draw[thick]  (-9, -1.2) -- (-9.455, -0.839);

            \fill[red] (-6.5, 0.8) circle (2pt) node[above] {};
            \fill[blue] (-7, 0) circle (2pt) node[left] {};
            \fill[red] (-6.5, -0.8) circle (2pt) node[below] {};

            \fill[red] (-5.5, 0.8) circle (2pt) node[above] {};
            \fill[red] (-5.5, -0.8) circle (2pt) node[left] {};
            \fill[blue] (-5, 0) circle (2pt) node[below] {};
            
            \fill (-6, 0.4) circle (2pt) node[left]{};
            \fill (-6, 1.2) circle (2pt) node[below] {};
            \fill (-6, -1.2) circle (2pt) node[below] {};

            \draw[red, thick] (-6.968, 0.058) -- (-6.5, 0.8);
            \draw[red, thick] (-6.968, -0.058) -- (-6.5, -0.8);

            \draw[green, thick] (-5.53, 0.75) -- (-6.47, -0.75);
            \draw[red, thick] (-5.5, 0.8) -- (-5.032, 0.058);
            \draw[red, thick](-5.5, -0.8) -- (-5.032, -0.058);
            \draw[green, thick] (-5.53, -0.75) -- (-6.47, 0.75);
            
            \draw[thick]  (-6, 0.4) -- (-6.478, -0.747);
            \draw[thick]  (-6, 0.4) -- (-5.545, 0.761);
            \draw[thick]  (-6, 1.2) -- (-5.545, 0.839);
            \draw[thick]  (-6, 1.2) -- (-6.455, 0.839);
            \draw[thick]  (-6, -1.2) -- (-5.518, 0.732);
            \draw[thick]  (-6, -1.2) -- (-6.455, -0.839);
           
        \end{tikzpicture}
        \caption{Construction \ref{cons:modify_cospectral_construction} applied to a graph obtained from Construction \ref{cons:A_cospectral_constructive}. Here, the red subgraphs correspond to $G^1$ and $G^2$ in Construction \ref{cons:A_cospectral_constructive}, with the $A$-cospectral vertices indicated in blue. The green edges are those that are added in applying Construction \ref{cons:modify_cospectral_construction} in two different ways. It is easy to see that the two graphs are not isomorphic (in the right graph the two vertices with degrees 4 and 5 are connected, whereas in the left graph they are not). By Lemma \ref{lem:modify_cospectral_construction}, the blue vertices are still $A$-cospectral in both graphs.}
        \label{fig:different_one_to_one}
    \end{figure}
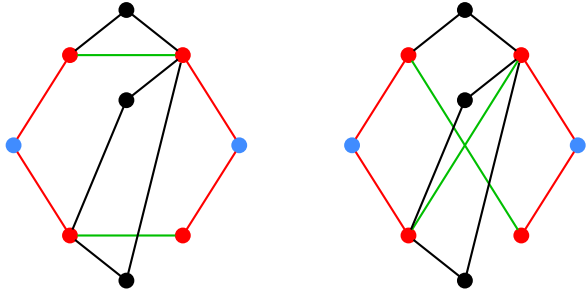

\begin{example}
Using Theorem \ref{thm:A_cospectral_constructive} and Lemma \ref{lem:modify_cospectral_construction} we can illuminate many examples of cospectral vertices and latent symmetries in the literature. Three such examples are provided in Figure \ref{fig:explain_literature_examples} with the subgraphs that correspond to $G^1$ and $G^2$ indicated in red and the $A$-cospectral vertices indicated in blue.

    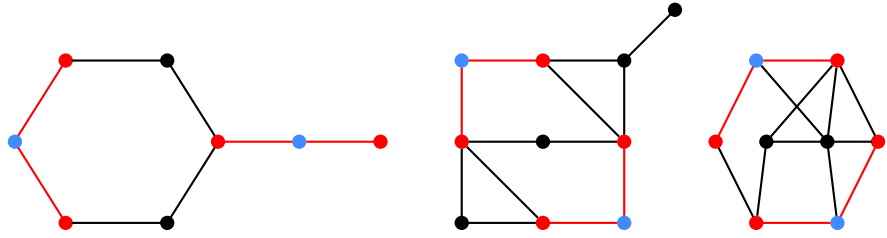
\begin{figure}[htb]
        \centering
      \begin{tikzpicture}[scale = 1.35]
            
           \fill[red] (-8.9, 0.8) circle (2pt) node[above] {};
\fill[blue] (-9.4, 0) circle (2pt) node[left] {};
\fill[red] (-8.9, -0.8) circle (2pt) node[below] {};

\fill[red] (-5.8, 0) circle (2pt) node[above] {};
\fill[red] (-7.4, 0) circle (2pt) node[left] {};
\fill[blue] (-6.6, 0) circle (2pt) node[below] {};

\fill (-7.9, 0.8) circle (2pt) node[left]{};
\fill (-7.9, -0.8) circle (2pt) node[below] {};

\draw[red, thick] (-9.368, 0.058) -- (-8.9, 0.8);
\draw[red, thick] (-9.368, -0.058) -- (-8.9, -0.8);

\draw[thick] (-7.935, 0.8) -- (-8.84, 0.8);
\draw[thick] (-7.9, 0.8) -- (-7.432, 0.058);
\draw[thick] (-7.9, -0.8) -- (-7.432, -0.058);
\draw[thick] (-7.935, -0.8) -- (-8.84, -0.8);

\draw[red, thick] (-6.67, 0) -- (-7.4, 0);
\draw[red, thick] (-6.53, 0) -- (-5.8, 0);

            \fill[red] (-5, 0) circle (2pt) node[above] {};
            \fill[blue] (-5, 0.8) circle (2pt) node[left] {};
            \fill[red] (-4.2, 0.8) circle (2pt) node[below] {};

            \fill[red] (-4.2, -0.8) circle (2pt) node[above] {};
            \fill[red] (-3.4,0) circle (2pt) node[left] {};
            \fill[blue] (-3.4, -0.8) circle (2pt) node[below] {};
            
            \fill (-5, -0.8) circle (2pt) node[left]{};
            \fill (-3.4, 0.8) circle (2pt) node[left]{};
            \fill (-2.9, 1.3) circle (2pt) node[left]{};
            \fill (-4.2, 0) circle (2pt) node[below] {};

            \draw[red, thick] (-5, 0.73) -- (-5, 0);
            \draw[red, thick] (-4.93, 0.8) -- (-4.2, 0.8);
            
            \draw[red, thick] (-3.4, -0.73) -- (-3.4, 0);
            \draw[red, thick] (-3.47, -0.8) -- (-4.2, -0.8);
            
            \draw[thick] (-5, -0.8) -- (-4.27, -0.8);
            \draw[thick] (-5, -0.8) -- (-5,-0.07);
            \draw[thick] (-4.2, 0) -- (-3.47, 0);
            \draw[thick] (-4.2, 0) -- (-4.93, 0);
            \draw[thick] (-3.4, 0.8) -- (-3.4, 0.07);
            \draw[thick] (-3.4, 0.8) -- (-4.13, 0.8);
            \draw[thick] (-3.4, 0.8) -- (-2.9, 1.3);
            \draw[thick] (-4.95,-0.05) -- (-4.25, -0.75);
            \draw[thick] (-4.15, 0.75) -- (-3.45, 0.05);

\fill[red] (-1.3, 0.8) circle (2pt) node[above] {};
\fill[blue] (-2.1, 0.8) circle (2pt) node[left] {};
\fill[red] (-2.5, 0) circle (2pt) node[below] {};

\fill[red] (-2.1, -0.8) circle (2pt) node[above] {};
\fill[red] (-0.9, 0) circle (2pt) node[left] {};
\fill[blue] (-1.3, -0.8) circle (2pt) node[below] {};

\fill (-2.0, 0) circle (2pt) node[left]{};
\fill (-1.4, 0) circle (2pt) node[left]{};

\draw[red, thick] (-2.13, 0.74) -- (-2.5, 0);
\draw[red, thick] (-2.03, 0.8) -- (-1.3, 0.8);

\draw[red, thick] (-1.27, -0.74) -- (-0.9, 0);
\draw[red, thick] (-1.37, -0.8) -- (-2.1, -0.8);

\draw[thick] (-1.31, -0.73) -- (-1.4, 0);
\draw[thick] (-2.05, 0.75) -- (-1.4, 0);
\draw[thick] (-0.96, 0) -- (-1.4, 0);
\draw[thick] (-1.31, 0.73) -- (-1.4, 0);
\draw[thick] (-2.09, -0.73) -- (-2.0, 0);
\draw[thick] (-1.35, 0.75) -- (-2.0, 0);
\draw[thick] (-2.47, -0.06) -- (-2.13, -0.74);
\draw[thick] (-1.27, 0.74) -- (-0.93, 0.06);
\draw[thick] (-2.0, 0) -- (-1.4, 0);

        \end{tikzpicture}
        \caption{All three graphs are from \cite{Hidden_symmetry}. In all of them, the blue vertex pairs are not symmetric but $A$-cospectral, and therefore latently symmetric. 
        The leftmost graph can be obtained by Construction \ref{cons:A_cospectral_constructive} and the other two can be obtained by  Construction \ref{cons:modify_cospectral_construction}. 
}
        \label{fig:explain_literature_examples}
    \end{figure}
\end{example}

Although all examples of $A$-cospectral vertices provided above are constructed with small subgraphs $G$ and $H$, there is no bound for the size of either $G$ or $H$ in Construction \ref{cons:A_cospectral_constructive}.

There is also a slightly different Laplacian version of Construction \ref{cons:A_cospectral_constructive} where instead of vertices of \(G^1\) and \(G^2\) connecting to vertices of \(H\), they connect to each other, respecting the orbit structure.

\begin{construction}[$L$-cospectral]\label{cons:L_cospectral_constructive}
Let \( G \) be a graph. Select a vertex in \( v_c \in V(G) \); we refer to it as the \textit{fixed vertex} of the construction.  
Divide the vertices of \( G \) into orbits under the action of \(\mathrm{Aut}(G, v_c)\). 
Now, take two copies of \( G \) and label them \( G^1(V^1,E^1) \) and \( G^2(V^2,E^2) \).  
(Denote the vertices of \( G^1 \) and \( G^2 \) as \( v_{j}^1, v_{k}^2, \) etc.) Finally, construct a new graph \( \widetilde{G} \) by
adding any number of edges between \( G^1 \) and \( G^2 \) subject to the following condition:  
Every edge connecting a vertex \( v_{j}^1 \in V(G^1) \) and a vertex \( v_{m}^2 \in V(G^2) \) should be between vertices of the same orbit; that is if $ (v_{j}^1,v_{m}^2) \in E(\widetilde{G})$, then we must have $v_{m} \in [v_j]_{v_c}$.
A schematic representation of this construction is shown in Figure \ref{fig:fig_L_cospectral_constructive}.
\end{construction}

    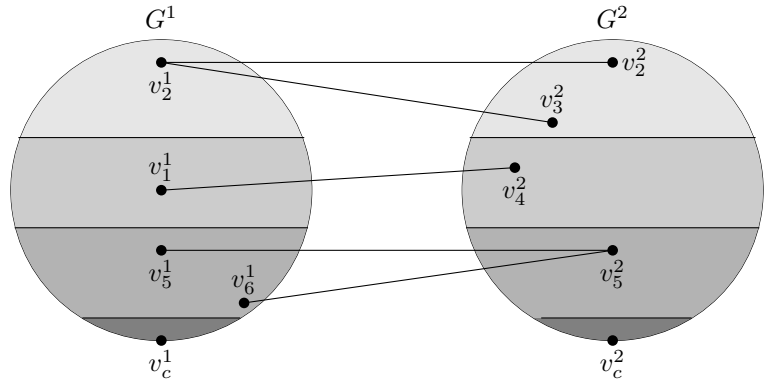
\begin{figure}[bth]
        \centering
        \begin{tikzpicture}
            \draw (0,0) circle (2cm);
            \node at (0, 2.3) {\( G^1 \)};
            
            \begin{scope}
                \clip (0,0) circle (2cm);
                \fill[gray!20] (-2, 0.7) rectangle (2, 2);
                \fill[gray!40] (-2, -0.5) rectangle (2, 0.7);
                \fill[gray!60] (-2, -1.7) rectangle (2, -0.5);
                \fill[gray!100] (-2, -2) rectangle (2, -1.7);
            \end{scope}
            \draw (-1.9, 0.7) -- (1.9, 0.7);
            \draw (-1.95, -0.5) -- (1.95, -0.5);
            \draw (-1.05, -1.7) -- (1.05, -1.7);
            
            \draw (6,0) circle (2cm);
            \node at (6, 2.3) {\( G^2 \)};
            
            \begin{scope}
                \clip (6,0) circle (2cm);
                \fill[gray!20] (4, 0.7) rectangle (8, 2);
                \fill[gray!40] (4, -0.5) rectangle (8, 0.7);
                \fill[gray!60] (4, -1.7) rectangle (8, -0.5);
                \fill[gray!100] (4, -2) rectangle (8, -1.7);
            \end{scope}
            \draw (4.1, 0.7) -- (7.9, 0.7);
            \draw (4.05, -0.5) -- (7.95, -0.5);
            \draw (5.05, -1.7) -- (7.05, -1.7);

            \fill (0, 1.7) circle (2pt) node[below] {\( v_{2}^1 \)};
            \fill (0, 0) circle (2pt) node[above] {\( v_{1}^1 \)};            
            \fill (0, -2) circle (2pt) node[below] {\( v_{c}^1 \)};
            \fill (0, -0.8) circle (2pt) node[below] {\( v_{5}^1 \)};
            \fill (1.1, -1.5) circle (2pt) node[above] {\( v_{6}^1 \)};
            
            \fill (6, 1.7) circle (2pt) node[right] {\( v_{2}^2 \)};
            \fill (5.2, 0.9) circle (2pt) node[above] {\( v_{3}^2 \)};
            \fill (4.7, 0.3) circle (2pt) node[below] {\( v_{4}^2 \)};
            \fill (6, -0.8) circle (2pt) node[below] {\( v_{5}^2 \)};            
            \fill (6, -2) circle (2pt) node[below] {\( v_{c}^2 \)};

            \draw (0, 0) -- (4.7, 0.3);
            \draw (0, -0.8) -- (6, -0.8);
            \draw (1.1, -1.5) -- (6, -0.8);
            \draw (5.2, 0.9) -- (0, 1.7);
            \draw (6, 1.7) -- (0, 1.7);
        \end{tikzpicture}
        \caption{Construction \ref{cons:L_cospectral_constructive} illustrated. The graphs \( G^1 \) and \( G^2 \) are divided into orbits under the action of \( \mathrm{Aut}(G,v_{c}) \). These orbits are shown by shades of gray, with the darkest gray containing only \(v_{c}\). The vertices \(v_{5}\), \(v_{6}\); \(v_{1}\), \(v_{4}\) and \(v_{2}\), \(v_{3}\) are in the same orbit. Every edge between vertices of \( G^1 \) and \( G^2 \) are between vertices belonging to the copy of the same orbit. By Theorem \ref{thm:L_cospectral_constructive}, the vertices \( v_{c}^1 \) and \( v_{c}^2 \) are $L$-cospectral in the graph \( \widetilde{G} \).}
        \label{fig:fig_L_cospectral_constructive}
    \end{figure}

\begin{theorem}\label{thm:L_cospectral_constructive}
Let $G$ be a graph and let $\widetilde{G}$ be obtained from $G$ via Construction \ref{cons:L_cospectral_constructive}, with the fixed vertex $v_c \in V(G)$
and subgraphs $G^1$, $G^2$. 
Then \( v_{c}^1 \) and \( v_{c}^2 \) are $L$-cospectral in \( \widetilde{G} \).  
\end{theorem}

\section{Constructing strongly cospectral vertices}\label{sec:strongly_cospectral_vertices}

In this section we show that the $A$-cospectral vertices of Construction \ref{cons:A_cospectral_constructive} are also strongly cospectral under certain conditions. 
We first require some preparatory results, the first of which is given in the following lemma.

\begin{lemma}\label{lem:eigenvector_same_on_orbit}
Let $G$ be a graph and let $v_i \in V(G)$ be a vertex of $G$. Consider the orbits $[v_j]_{v_i}$ of $G$ under the action of the group $\mathrm{Aut}(G, v_i)$. Consider also the standard basis vector $\boldsymbol{v}_{i}$ associated with the vertex $v_i$ and the spectral decomposition of $\boldsymbol{v}_{i}$ with respect to the eigenvectors $\boldsymbol{w}_\lambda$ of the adjacency matrix $A(G)$:
\[
\boldsymbol{v}_{i} = \sum_{\lambda \in \Lambda_A} c_\lambda \boldsymbol{w}_\lambda.
\]
Then $(\boldsymbol{w}_\lambda)_{v_\ell} = (\boldsymbol{w}_\lambda)_{v_m}$ for all $\lambda \in \Lambda_A$ such that $c_\lambda \neq 0$ and for all vertices $v_\ell$, $v_m$ in the same orbit, that is, for all $v_\ell$, $v_m \in V(G)$ satisfying $v_m \in [v_\ell]_{v_i}$.
\end{lemma}

Besides being useful later in the paper, the following theorem can be of interest in itself.

\begin{theorem}\label{thm:eigenvalue_of_subgraph_also_of_graph}
Let $G$ and $H$ be two graphs and let $\widetilde{G}$ be a graph constructed from them via Construction \ref{cons:A_cospectral_constructive}, with the fixed vertex $v_c \in V(G)$. 
Consider the vertex $v_c$ in the graph $G$. The standard basis vector $\boldsymbol{v}_{c}$ associated with this vertex has a spectral decomposition with respect to the adjacency matrix $A(G)$ as
\[
\boldsymbol{v}_{c} = \sum_{\lambda \in \Lambda_A} c_\lambda \, \boldsymbol{w}_\lambda.
\]
If $c_\lambda \neq 0$ in this spectral decomposition, then $\lambda$ is also an eigenvalue of $A(\widetilde{G})$. 
\end{theorem}

\begin{proof}
We know that the vertices of $\widetilde{G}$ can be partitioned into the vertices of the subgraphs $G^1$, $G^2$ and $H$. Let $\lambda\in \Lambda_A$ be an eigenvalue of $A$ such that $c_{\lambda}\neq 0$ in the spectral decomposition of $\boldsymbol{v}_{c}$. Consider the vector $\widetilde{\boldsymbol{w}}_\lambda \in \mathbb{R}^{|V(\widetilde{G})|}$ associated with the eigenvector $\boldsymbol{w}_\lambda$ of $A(G)$, defined for $v\in V(\widetilde{G})$ by
    \[
(\widetilde{\boldsymbol{w}}_\lambda)_v := 
\begin{cases}
0 & \text{if } v = v'_\alpha \in V(H) \\
(\boldsymbol{w}_\lambda)_{v_k} & \text{if } v = v_k^1 \in V(G^1) \\
-(\boldsymbol{w}_\lambda)_{v_k} & \text{if } v= v_k^2 \in V(G^2)
\end{cases}
\]
We will show that $\widetilde{\boldsymbol{w}}_\lambda$ is an eigenvector of $\widetilde{A}=A(\widetilde{G})$. For vertices $v'_\alpha \in V(H)$ of $\widetilde{G}$, we have
\[
(\widetilde{A} \widetilde{\boldsymbol{w}}_\lambda)_{v'_\alpha} = \sum_{\substack{v_k^1 \sim {v'_\alpha} \\ v_k \in V(G)}} (\widetilde{\boldsymbol{w}}_\lambda)_{v_k^1} + \sum_{\substack{v_\ell^2 \sim {v'_\alpha} \\ v_\ell \in V(G)}} (\widetilde{\boldsymbol{w}}_\lambda)_{v_\ell^2} + \sum_{\substack{{v'_\beta} \sim {v'_\alpha} \\ {v'_\beta}\in V(H)}} (\widetilde{\boldsymbol{w}}_\lambda)_{v'_\beta} 
\]
\[
= \sum_{\substack{v_k^1 \sim {v'_\alpha} \\ v_k \in V(G)}} (\boldsymbol{w}_\lambda)_{v_k} - \sum_{\substack{v_\ell^2 \sim {v'_\alpha} \\ v_\ell \in V(G)}} (\boldsymbol{w}_\lambda)_{v_\ell} 
\]
where the last equality follows from definition of $\widetilde{\boldsymbol{w}}_\lambda$. By Construction \ref{cons:A_cospectral_constructive}, vertices $v'_\alpha \in V(H)$ have equal number of edges connecting it to corresponding orbits of $G^1$ and $G^2$. In other words, for every edge between vertices $v'_\alpha \in V(H)$ and $v_k^1 \in V(G^1)$, there exists a corresponding edge between $v'_\alpha$ and a vertex $v_\ell^2 \in V(G^2)$ such that $v_\ell \in [v_k]_{v_c}$, and vice versa. This fact together with Lemma \ref{lem:eigenvector_same_on_orbit} imply that
\[
(\widetilde{A} \widetilde{\boldsymbol{w}}_\lambda)_{v'_\alpha} = \sum_{\substack{v_k^1 \sim {v'_\alpha} \\ v_k \in V(G)}} (\boldsymbol{w}_\lambda)_{v_k} - \sum_{\substack{v_\ell^2 \sim {v'_\alpha} \\ v_\ell \in V(G)}} (\boldsymbol{w}_\lambda)_{v_\ell} = 0 = \lambda(\widetilde{\boldsymbol{w}}_\lambda)_{v'_\alpha}.
\]
For vertices $v_k^1 \in V(G^1)$ we have
\[
(\widetilde{A} \widetilde{\boldsymbol{w}}_\lambda)_{v_k^1} = \sum_{\substack{v_\ell \sim v_k \\ v_\ell \in V(G)}} (\widetilde{\boldsymbol{w}}_\lambda)_{v_\ell^1} + \sum_{\substack{v'_\alpha \sim v_k^1 \\ v'_\alpha\in V(H)}} (\widetilde{\boldsymbol{w}}_\lambda)_{v'_\alpha} = \sum_{\substack{v_\ell \sim v_k \\ v_\ell \in V(G)}} (\boldsymbol{w}_\lambda)_{v_\ell} 
\]\[
=  \lambda (\boldsymbol{w}_\lambda)_{v_k} = \lambda (\widetilde{\boldsymbol{w}}_\lambda)_{v_k^1}
\]
where the second and fourth equalities are by the definition of $\widetilde{\boldsymbol{w}}_{\lambda}$ and the third equality is due to $\boldsymbol{w}_{\lambda}$ being an eigenvector of $A(G)$ with eigenvalue $\lambda$. A similar argument works for vertices of $G^2$. Hence, 
\[
\widetilde{A} \widetilde{\boldsymbol{w}}_\lambda = \lambda \widetilde{\boldsymbol{w}}_\lambda
\]
and so $\lambda$ is an eigenvalue of $\widetilde{A}$, as claimed.  
\end{proof}

\begin{definition}

We call the eigenvalues of $\widetilde{A}$ considered in Theorem \ref{thm:eigenvalue_of_subgraph_also_of_graph} the \textit{eigenvalues of $\widetilde{A}$ induced by $G$} and denote the set of such eigenvalues by $\Lambda_{G}^{\widetilde{A}}$.
The normalized versions of the eigenvectors $\widetilde{\boldsymbol{w}}_\lambda$ constructed in the proof of Theorem \ref{thm:eigenvalue_of_subgraph_also_of_graph} are similarly called \textit{eigenvectors of $\widetilde{A}$ induced by $G$}.   
    
\end{definition}

\begin{lemma}\label{lem:v1-v2_contains_only_induced_eigenvectors}
Let $G$ and $H$ be two graphs. Let $\widetilde{G}$ be a graph constructed as in Construction \ref{cons:A_cospectral_constructive}, with the fixed vertex $v_c \in V(G)$ and subgraphs $G^1$, $G^2$, $H$. Consider the adjacency matrix $\widetilde{A}$, its eigenvectors, and the vector $\boldsymbol{v}_{c}^1 - \boldsymbol{v}_{c}^2$. The spectral decomposition 
    \[
    \boldsymbol{v}_{c}^1 - \boldsymbol{v}_{c}^2 = \sum_{\lambda \in \Lambda_{\widetilde{A}}} c_{\lambda} \boldsymbol{y}_\lambda
    \]
of $\boldsymbol{v}_{c}^1 - \boldsymbol{v}_{c}^2$ contains precisely those eigenvectors of $\widetilde{A}$ that are induced by $G$. That is, if $c_{\lambda} \neq 0$, then $\lambda$ is an eigenvalue of $\widetilde{A}$ induced by $G$ and the eigenvector $\boldsymbol{y}_\lambda$ equals the eigenvector $ \widetilde{\boldsymbol{w}}_{\lambda}$ of $\widetilde{A}$ induced by $G$.
\end{lemma}

The next result shows that one only needs to check the eigenvalues of $\widetilde{A}$ induced by $G$ to ensure strong cospectrality of the $A$-cospectral vertices obtained by Construction \ref{cons:A_cospectral_constructive}.

\begin{theorem}\label{thm:almost_strongly_cospectral}
Let $G$ and $H$ be two graphs and let $\widetilde{G}$ be a graph obtained from them via Construction \ref{cons:A_cospectral_constructive}, with the fixed vertex $v_c \in V(G)$.
Suppose that eigenvalues of $\widetilde{A}$ induced by $G$ are all simple eigenvalues of $\widetilde{A}$. Then the vertices $v_c^1$ and $v_c^2$ are strongly cospectral.
\end{theorem}

\begin{proof}
Consider the $A$-cospectral vertices $v_c^1$ and $v_c^2$. Now, with respect to the adjacency matrix $\widetilde{A}=A(\widetilde{G})$, $\boldsymbol{v}_{c}^1$ and $\boldsymbol{v}_{c}^2$ have the spectral decompositions
\begin{equation}\label{eq:strongly_cospectral_1}
    \boldsymbol{v}_{c}^1 = \sum_{\lambda \in \Lambda_{G}^{\widetilde{A}}} \alpha_\lambda^1 \, \widetilde{\boldsymbol{w}}_\lambda + \sum_{\lambda \in X} \beta_\lambda^1 \, \boldsymbol{y}_\lambda^1 \quad \text{and} \quad \boldsymbol{v}_{c}^2 = \sum_{\lambda \in \Lambda_{G}^{\widetilde{A}}} \alpha_\lambda^2 \, \widetilde{\boldsymbol{w}}_\lambda + \sum_{\lambda \in X} \beta_\lambda^2 \, \boldsymbol{y}_\lambda^2
\end{equation}
where $\widetilde{\boldsymbol{w}}_\lambda$ are the eigenvectors of $\widetilde{A}$ induced by $G$; $\boldsymbol{y}_\lambda^1$, $\boldsymbol{y}_\lambda^2$ denote other eigenvectors of $\widetilde{A}$ and  
\[
X := \left( \Lambda_{\widetilde{A}} \setminus \Lambda_{G}^{\widetilde{A}} \right) \cup \left\{ \lambda \in \Lambda_{G}^{\widetilde{A}} \mid \mu_{\widetilde{A}}(\lambda) > 1 \right\}
\]
denotes the set of eigenvalues of $\widetilde{A}$ not induced by $G$. Here $\mu_{\widetilde{A}}(\lambda)$ is the multiplicity of $\lambda$ as an eigenvalue of $\widetilde{A}$. Note that $X$ potentially includes some of the eigenvalues of $\widetilde{A}$ induced by $G$ if they have multiplicity $\mu_{\widetilde{A}}(\lambda) > 1$ as eigenvalues of $\widetilde{A}$. If this is the case, then in (\ref{eq:strongly_cospectral_1}) the eigenvectors $\boldsymbol{y}_{\lambda}^1$ and $\boldsymbol{y}_{\lambda}^2$ corresponding to $\lambda$ are chosen such that they are orthogonal to the eigenvector $\widetilde{\boldsymbol{w}}_{\lambda}$.

Now, the vector $\boldsymbol{v}_{c}^1 - \boldsymbol{v}_{c}^2$ also has a spectral decomposition, and by Lemma \ref{lem:v1-v2_contains_only_induced_eigenvectors} it should contain only eigenvectors $\widetilde{\boldsymbol{w}}_\lambda$ of $\widetilde{A}$ induced by $G$, i.e., it has the form
\[
\boldsymbol{v}_{c}^1 - \boldsymbol{v}_{c}^2 = \sum_{\lambda \in \Lambda_{G}^{\widetilde{A}}} (\alpha_\lambda^1 - \alpha_\lambda^2) \, \widetilde{\boldsymbol{w}}_\lambda.
\]
By (\ref{eq:strongly_cospectral_1}), this directly implies
\[
\beta_\lambda := \beta_\lambda^1 = \beta_\lambda^2 \quad \text{and} \quad \boldsymbol{y}_\lambda := \boldsymbol{y}_\lambda^1 = \boldsymbol{y}_\lambda^2, \quad  \forall \lambda \in X.
\]
(The latter was not obvious as $\boldsymbol{y}_\lambda^1$ and $\boldsymbol{y}_\lambda^2$ could have been different eigenvectors of the same eigenvalue if the eigenspace of $\lambda$ is more than one dimensional.) Moreover,  
\[
\alpha_\lambda := \alpha_\lambda^1 = \widetilde{\boldsymbol{w}}_\lambda^T \boldsymbol{v}_{c}^1 = -\widetilde{\boldsymbol{w}}_\lambda^T \boldsymbol{v}_{c}^2 = -\alpha_\lambda^2, 
\quad \forall\lambda \in \Lambda_{G}^{\widetilde{A}},
\]
where the equalities follow from (\ref{eq:strongly_cospectral_1}) and the definition of $\widetilde{\boldsymbol{w}}_\lambda$, scaled to be a unit vector. Now, for  $\lambda \in \Lambda_{\widetilde{A}}$, let $E_\lambda$ be the orthogonal projection map onto the eigenspace of $\widetilde{A}$ corresponding to $\lambda$. Then, 
\begin{alignat*}{2}
E_\lambda \boldsymbol{v}_{c}^1 &= E_\lambda \boldsymbol{v}_{c}^2 = \beta_\lambda \boldsymbol{y}_\lambda  &\quad  &\text{if } \lambda \in X \text{ and } \lambda \notin \Lambda_{G}^{\widetilde{A}},
\\
E_\lambda \boldsymbol{v}_{c}^1 &= -E_\lambda \boldsymbol{v}_{c}^2 = \alpha_\lambda \widetilde{\boldsymbol{w}}_\lambda  & \quad & \text{if } \lambda \notin X \text{ and } \lambda \in \Lambda_{G}^{\widetilde{A}}.
\end{alignat*}
This implies that if \( X \cap \Lambda_{G}^{\widetilde{A}} = \emptyset \), i.e., if the eigenvalues of $\widetilde{A}$ induced by $G$ are all simple eigenvalues of $\widetilde{A}$, then $v_c^1$ and $v_c^2$ are strongly cospectral in $\widetilde{G}$.  
\end{proof}

\begin{remark}\label{rem:strongly_cospectral}

Theorem \ref{thm:almost_strongly_cospectral} fails to work if there exists $\lambda \in X  \cap \Lambda_{G}^{\widetilde{A}}$, or equivalently if there exists $\lambda \in \Lambda_G^{\widetilde{A}}$ which has multiplicity $\mu_{\widetilde{A}}(\lambda) > 1$ as an eigenvalue of $\widetilde{A}$. In such a case, we have $E_\lambda \boldsymbol{v}_{c}^1 = \beta_\lambda \boldsymbol{y}_\lambda + \alpha_\lambda \widetilde{\boldsymbol{w}}_\lambda $ and $E_\lambda \boldsymbol{v}_{c}^2 = \beta_\lambda \boldsymbol{y}_\lambda - \alpha_\lambda \widetilde{\boldsymbol{w}}_\lambda$ and so these vectors, in general, become linearly independent. However, since Construction \ref{cons:A_cospectral_constructive} is preserved under modifications of the subgraph $H$, one might be able to get rid of such eigenvalue degeneracies while preserving the cospectrality of $v_c^1$ and $v_c^2$. An operation that might help such modifications is given in Lemma \ref{lem:reduce_eigval_multiplicity}.
    
\end{remark}

\begin{lemma}\label{lem:reduce_eigval_multiplicity}
    
Let $G$ be a graph such that $\lambda$ is an eigenvalue of $A(G)$ of multiplicity $\mu_{A}(\lambda) = \ell > 1$. Let $\boldsymbol{y}_\lambda$ be an eigenvector of $A(G)$ associated with $\lambda$. We have $(\boldsymbol{y}_\lambda)_{v_m} \neq 0$ for some $v_m \in V(G)$. Let $\widehat{G}$ be the graph obtained from $G$ by attaching a single vertex $\hat{v}$ to $v_m$. Then, the multiplicity $\mu_{\widehat{A}}(\lambda)$ of $\lambda$ as an eigenvalue of $A(\widehat{G})$ equals $\ell - 1$.
\end{lemma}

We defer the remainder of the proofs to the Appendix.

\subsubsection*{Acknowledgment}

O.E.E. thanks TÜBİTAK for the scholarship support (BİDEB 2211-Domestic Graduate Scholarship Programme) given for this work.

\appendix
\section{Proofs}

\begin{proof}[Proof of Lemma \ref{lem:modify_cospectral_construction}]
In the proof of Theorem \ref{thm:A_cospectral_constructive} we laid three separate claims and proved them together inductively. It is easy to see that the base case for all three claims are unaffected by the operation of connecting orbits. It is also easy to see that the inductive step for claim \ref{thm1:claim1} is also unaffected.
The inductive steps of claims \ref{thm1:claim2} and \ref{thm1:claim3} also continue to hold; since in the induction step, by the definition of the operation of connecting orbits and using the induction hypothesis for claims \ref{thm1:claim2} and \ref{thm1:claim3}, the newly added edges contribute the same term to both sides of the equalities in claims \ref{thm1:claim2} and \ref{thm1:claim3}. Thus, the proof of Theorem \ref{thm:A_cospectral_constructive} applies and hence, $v_c^1$ and $v_c^2$ are $A$-cospectral in $\widehat{G}$.  
\end{proof}

\begin{proof}[Proof of Theorem \ref{thm:L_cospectral_constructive}]

We claim that the following hold for all $k \in \mathbb{N}$:

    \begin{enumerate}[label=(\roman*)]
        \item\label{thm2:claim1}  $\left( \widetilde{L}^k (\boldsymbol{v}_{c}^1 + \boldsymbol{v}_{c}^2) \right)_{v_{m}^1}  
= \left( \widetilde{L}^k (\boldsymbol{v}_{c}^1 + \boldsymbol{v}_{c}^2) \right)_{v_{m}^2}$,
        \item\label{thm2:claim2}  $\left( \widetilde{L}^k (\boldsymbol{v}_{c}^1 + \boldsymbol{v}_{c}^2) \right)_{v_{m}^1}  
= \left( \widetilde{L}^k (\boldsymbol{v}_{c}^1 + \boldsymbol{v}_{c}^2) \right)_{v_{j}^1}$, \( \forall v_m, v_j \in V(G) \) such that \mbox{\( v_j \in [v_m]_{v_c}  \)}.
    \end{enumerate}

Clearly, claims \ref{thm2:claim1} and \ref{thm2:claim2} together imply $( \widetilde{L}^k (\boldsymbol{v}_{c}^1 + \boldsymbol{v}_{c}^2) )_{v_{m}^2}  
= ( \widetilde{L}^k (\boldsymbol{v}_{c}^1 + \boldsymbol{v}_{c}^2) )_{v_{j}^2}$ as well. We will prove by induction.

\textit{Base Case \( k = 0 \)}: Trivial.

\textit{Inductive Step}: Suppose claims \ref{thm2:claim1} and \ref{thm2:claim2} hold for $k - 1$. We will show they also hold for $k$. To prove claim \ref{thm2:claim1}, we have for any $v_m \in V(G)$,  
\begin{align*}
\left( \widetilde{L}^k (\boldsymbol{v}_{c}^1 + \boldsymbol{v}_{c}^2) \right)_{v_{m}^1} &= d_{v_m} \left( \widetilde{L}^{k-1} (\boldsymbol{v}_{c}^1 + \boldsymbol{v}_{c}^2) \right)_{v_{m}^1}  
- \sum\limits_{v_j \sim v_m} \left( \widetilde{L}^{k-1} (\boldsymbol{v}_{c}^1 + \boldsymbol{v}_{c}^2) \right)_{v_{j}^1}
\\
& \hspace{1cm} + \sum\limits_{\substack{v_\ell \in V(G) \\ v_{\ell}^2 \sim v_{m}^1}}  
\left[ \left( \widetilde{L}^{k-1} (\boldsymbol{v}_{c}^1 + \boldsymbol{v}_{c}^2) \right)_{v_{m}^1}  
- \left( \widetilde{L}^{k-1} (\boldsymbol{v}_{c}^1 + \boldsymbol{v}_{c}^2) \right)_{v_{\ell}^2} \right]
\\
&= d_{v_m} \left( \widetilde{L}^{k-1} (\boldsymbol{v}_{c}^1 + \boldsymbol{v}_{c}^2) \right)_{v_{m}^2}  
- \sum\limits_{v_j \sim v_m} \left( \widetilde{L}^{k-1} (\boldsymbol{v}_{c}^1 + \boldsymbol{v}_{c}^2) \right)_{v_{j}^2} 
\end{align*}
where we used the induction hypothesis for claims \ref{thm2:claim1} and \ref{thm2:claim2} along with the construction fact that if $v_{\ell}^2 \sim v_{m}^1$ then $v_\ell \in [v_m]_{v_c}$. In the same way, we obtain 
\begin{align*}
& d_{v_m} \left( \widetilde{L}^{k-1} (\boldsymbol{v}_{c}^1 + \boldsymbol{v}_{c}^2) \right)_{v_{m}^2}  
- \sum\limits_{v_j \sim v_m} \left( \widetilde{L}^{k-1} (\boldsymbol{v}_{c}^1 + \boldsymbol{v}_{c}^2) \right)_{v_{j}^2}
\\
& = d_{v_m} \left( \widetilde{L}^{k-1} (\boldsymbol{v}_{c}^1 + \boldsymbol{v}_{c}^2) \right)_{v_{m}^2}  
- \sum\limits_{v_j \sim v_m} \left( \widetilde{L}^{k-1} (\boldsymbol{v}_{c}^1 + \boldsymbol{v}_{c}^2) \right)_{v_{j}^2}
\\
& \hspace{1cm} + \sum\limits_{\substack{v_\ell \in V(G) \\ v_{\ell}^1 \sim v_{m}^2}}\left[ \left( \widetilde{L}^{k-1} (\boldsymbol{v}_{c}^1 + \boldsymbol{v}_{c}^2) \right)_{v_{m}^2}  
- \left( \widetilde{L}^{k-1} (\boldsymbol{v}_{c}^1 + \boldsymbol{v}_{c}^2) \right)_{v_{\ell}^1} \right] 
\\
&= \left( \widetilde{L}^k (\boldsymbol{v}_{c}^1 + \boldsymbol{v}_{c}^2) \right)_{v_{m}^2}
\end{align*}
proving claim \ref{thm2:claim1}. Now to prove claim \ref{thm2:claim2}, consider two vertices \( v_m, v_j \in V(G) \) such that \( v_j \in [v_m]_{v_c} \). 
We have  
\begin{align} \label{eq:thm2:eq1}
\left( \widetilde{L}^k (\boldsymbol{v}_{c}^1 + \boldsymbol{v}_{c}^2) \right)_{v_{m}^1}
& = d_{v_m} \left( \widetilde{L}^{k-1} (\boldsymbol{v}_{c}^1 + \boldsymbol{v}_{c}^2) \right)_{v_{m}^1}
- \sum\limits_{v_\ell \sim v_m} \left( \widetilde{L}^{k-1} (\boldsymbol{v}_{c}^1 + \boldsymbol{v}_{c}^2) \right)_{v_{\ell}^1}
 \nonumber \\
& \hspace{1cm} + \sum\limits_{v_{r} \in V(G)  \atop v_{r}^2 \sim v_{m}^1}  
\left[ \left( \widetilde{L}^{k-1} (\boldsymbol{v}_{c}^1 + \boldsymbol{v}_{c}^2) \right)_{v_{m}^1}  
- \left( \widetilde{L}^{k-1} (\boldsymbol{v}_{c}^1 + \boldsymbol{v}_{c}^2) \right)_{v_{r}^2} \right]
 \nonumber \\
& = d_{v_m} \left( \widetilde{L}^{k-1} (\boldsymbol{v}_{c}^1 + \boldsymbol{v}_{c}^2) \right)_{v_{m}^1}  
- \sum\limits_{v_\ell \sim v_m} \left( \widetilde{L}^{k-1} (\boldsymbol{v}_{c}^1 + \boldsymbol{v}_{c}^2) \right)_{v_{\ell}^1}    
\end{align}
where similar to claim \ref{thm2:claim1}, we used the induction hypothesis for claims \ref{thm2:claim1} and \ref{thm2:claim2} along with the fact that if $v_{r}^2 \sim v_{m}^1$ then $v_r \in [v_m]_{v_c}$. Now, since \( v_j \) and \( v_m \) are in the same  
automorphism orbit, their neighbors are also  
in the same automorphism orbits. 
Hence, by the induction hypothesis for claim \ref{thm2:claim2},  (\ref{eq:thm2:eq1}) becomes
\begin{align*}
\left( \widetilde{L}^k (\boldsymbol{v}_{c}^1 + \boldsymbol{v}_{c}^2) \right)_{v_{m}^1}
&= d_{v_m} \left( \widetilde{L}^{k-1} (\boldsymbol{v}_{c}^1 + \boldsymbol{v}_{c}^2) \right)_{v_{m}^1}  
- \sum\limits_{v_\ell \sim v_m} \left( \widetilde{L}^{k-1} (\boldsymbol{v}_{c}^1 + \boldsymbol{v}_{c}^2) \right)_{v_{\ell}^1}
\\
&= d_{v_j} \left( \widetilde{L}^{k-1} (\boldsymbol{v}_{c}^1 + \boldsymbol{v}_{c}^2) \right)_{v_{j}^1}  
- \sum\limits_{v_s \sim v_j} \left( \widetilde{L}^{k-1} (\boldsymbol{v}_{c}^1 + \boldsymbol{v}_{c}^2) \right)_{v_{s}^1}
\\
&= d_{v_j} \left( \widetilde{L}^{k-1} (\boldsymbol{v}_{c}^1 + \boldsymbol{v}_{c}^2) \right)_{v_{j}^1}  
- \sum\limits_{v_s \sim v_j} \left( \widetilde{L}^{k-1} (\boldsymbol{v}_{c}^1 + \boldsymbol{v}_{c}^2) \right)_{v_{s}^1}
\\
& \hspace{2cm} + \sum\limits_{v_{r} \in V(G)  \atop v_{r}^2 \sim v_{j}^1}  
\left[ \left( \widetilde{L}^{k-1} (\boldsymbol{v}_{c}^1 + \boldsymbol{v}_{c}^2) \right)_{v_{j}^1}  
- \left( \widetilde{L}^{k-1} (\boldsymbol{v}_{c}^1 + \boldsymbol{v}_{c}^2) \right)_{v_{r}^2} \right] 
\\
&= \left( \widetilde{L}^k (\boldsymbol{v}_{c}^1 + \boldsymbol{v}_{c}^2) \right)_{v_{j}^1}
\end{align*}
where in the second equality we used the fact that  
\( d_{v_m} = d_{v_j} \) and in the  
third equality we used the induction hypothesis for claims \ref{thm2:claim1} and \ref{thm2:claim2} along with the fact that if $v_{r}^2 \sim v_{j}^1$ then $v_r \in [v_j]_{v_c}$, similar to before. Hence, we have proved claim \ref{thm2:claim2} and so the induction argument is completed. Now, we have shown in particular that $\left( \widetilde{L}^k (\boldsymbol{v}_{c}^1 + \boldsymbol{v}_{c}^2) \right)_{v_{c}^1}  
= \left( \widetilde{L}^k (\boldsymbol{v}_{c}^1 + \boldsymbol{v}_{c}^2) \right)_{v_{c}^2}$,  \(\forall k \in \mathbb{N} \). Equivalently, 
\[
\left( \widetilde{L}^k (\boldsymbol{v}_{c}^1 + \boldsymbol{v}_{c}^2) \right)^T (\boldsymbol{v}_{c}^1 - \boldsymbol{v}_{c}^2) = 0 \quad \forall  k \in \mathbb{N}.
\]
Hence, \(\langle \boldsymbol{v}_{c}^1 + \boldsymbol{v}_{c}^2 \rangle_{\widetilde{L}} \perp \langle \boldsymbol{v}_{c}^1 - \boldsymbol{v}_{c}^2 \rangle_{\widetilde{L}}\), and so \( v_{c}^1 \) and \( v_{c}^2 \) are $L$-cospectral in \(\widetilde{G}\).
\end{proof}

\begin{proof}[Proof of Lemma \ref{lem:eigenvector_same_on_orbit}]
Since $v_\ell$ and $v_m$ are in the same orbit $[v_\ell]_{v_i}$, 
\begin{align}\label{eq:lem1:eq1}
(A^k \boldsymbol{v}_{i})_{v_\ell} = (A^k \boldsymbol{v}_{i})_{v_m} \quad \forall k \in \mathbb{N}.    
\end{align}
Now, $(A^k \boldsymbol{v}_{i})_{v_\ell}$ and $(A^k \boldsymbol{v}_{i})_{v_m}$ can be written as 
\[
(A^k \boldsymbol{v}_{i})_{v_\ell} = \sum_{\lambda \in \Lambda_A} \lambda^k c_\lambda (\boldsymbol{w}_\lambda)_{v_\ell} \quad \text{and} \quad (A^k \boldsymbol{v}_{i})_{v_m} = \sum_{\lambda \in \Lambda_A} \lambda^k c_\lambda (\boldsymbol{w}_\lambda)_{v_m}.
\]
Therefore, (\ref{eq:lem1:eq1}) is equivalent to
\[
\sum_{\lambda \in \Lambda_A} \lambda^k c_\lambda (\boldsymbol{w}_\lambda)_{v_\ell} = \sum_{\lambda \in \Lambda_A} \lambda^k c_\lambda (\boldsymbol{w}_\lambda)_{v_m} \quad \forall k \in \mathbb{N},
\]
which implies that, for all $ \lambda \in \Lambda_A $ such that $ c_\lambda \neq 0$, 
\[
(\boldsymbol{w}_\lambda)_{v_\ell} = (\boldsymbol{w}_\lambda)_{v_m}. \qedhere
\]
\end{proof}

\begin{proof}[Proof of Lemma \ref{lem:v1-v2_contains_only_induced_eigenvectors}]
Let $\boldsymbol{v}_{c} = \sum_{\lambda \in \Lambda_A} c_\lambda \, \boldsymbol{w}_\lambda$ be the spectral decomposition of $\boldsymbol{v}_{c}$ with respect to $A(G)$ and let $\widetilde{\boldsymbol{w}}_{\lambda}$ denote eigenvectors of $\widetilde{A}$ induced by $G$. Consider the vector $\sum_{\lambda \in \Lambda_A} c_\lambda \, \widetilde{\boldsymbol{w}}_\lambda$. We will show that $\boldsymbol{v}_{c}^1 - \boldsymbol{v}_{c}^2 = \sum_{\lambda \in \Lambda_A} c_\lambda \, \widetilde{\boldsymbol{w}}_\lambda$. 
     
Let $v'_\alpha\in V(H)$. Then $(\sum_{\lambda \in \Lambda_A} c_\lambda \, \widetilde{\boldsymbol{w}}_\lambda)_{v'_\alpha} = \sum_{\lambda \in \Lambda_A} c_\lambda \, (\widetilde{\boldsymbol{w}}_\lambda)_{v'_\alpha} = 0$ since $(\widetilde{\boldsymbol{w}}_{\lambda})_{v'_\alpha} = 0$ for all eigenvectors of $\widetilde{A}$ induced by $G$. 

Let $v_k^1 \in V(G^1)$. Then 
\begin{align*}
\big(\sum_{\lambda \in \Lambda_A} c_\lambda \, \widetilde{\boldsymbol{w}}_\lambda\big)_{v_k^1} &= \sum_{\lambda \in \Lambda_A} c_\lambda \, (\widetilde{\boldsymbol{w}}_\lambda)_{v_k^1} = \sum_{\lambda \in \Lambda_A} c_\lambda \, (\boldsymbol{w}_\lambda)_{v_k} 
\\
&= \big(\sum_{\lambda \in \Lambda_A} c_\lambda \, \boldsymbol{w}_\lambda\big)_{v_k} = (\boldsymbol{v}_{c})_{v_k} = \begin{cases}
    0 \quad \text{ if } v_k \neq v_c\\
    1 \quad \text{ if } v_k = v_c
\end{cases}
\end{align*}
Similarly, we have 
\[
\big(\sum_{\lambda \in \Lambda_A} c_\lambda \, \widetilde{\boldsymbol{w}}_\lambda\big)_{v_k^2} = \begin{cases}
  0  & \text{if } v_k \neq v_c \\
 -1  & \text{if } v_k = v_c. 
\end{cases}
\]
Hence, $ \sum_{\lambda \in \Lambda_A} c_\lambda \, \widetilde{\boldsymbol{w}}_\lambda = \boldsymbol{v}_{c}^1 - \boldsymbol{v}_{c}^2$. By Theorem \ref{thm:eigenvalue_of_subgraph_also_of_graph}, if $c_{\lambda} \neq 0$ in the spectral decomposition of $\boldsymbol{v}_{c}$, then $\lambda \in \Lambda_A$ and $\widetilde{\boldsymbol{w}}_{\lambda}$  are eigenvalues and eigenvectors of $\widetilde{A}$. Because the spectral decomposition of a vector is unique, this implies that $\sum_{\lambda \in \Lambda_A} c_\lambda \, \widetilde{\boldsymbol{w}}_\lambda$ is the spectral decomposition of $\boldsymbol{v}_{c}^1 - \boldsymbol{v}_{c}^2$ with respect to $\widetilde{A}$.
\end{proof}

\begin{proof}[Proof of Lemma \ref{lem:reduce_eigval_multiplicity}]
Let $n = |V(G)|$. We may suppose that the vertices of the graph $\widehat{G}$ are indexed in the same order as those of $G$ with the exception that the extra vertex $\hat{v}$ is the final $(n+1)$-th vertex. Hence, $A(\widehat{G})$ can be assumed to have the form  
\[
\widehat{A}=\left[
\begin{array}{ccccc}
&  &          &    &               0 \\
&  &          &    &         \vdots       \\
    &      &    A   &     &       1     \\
       &        &    &      &    \vdots   \\
     0  &      \cdots  &  1  &    \cdots  &    0           
\end{array}
\right]
\]
where the non-zero entries in the final row and column are in the $(\hat{v},v_m)$ and $(v_m,\hat{v})$ entries.

We will follow the proof of Theorem 9.5.1 of \cite{godsil_algebraic_graph_theory} with slight modifications. Let $\lambda_i(A)$ and $\lambda_i(\widehat{A})$ denote the $i$-th largest eigenvalue of $A$ and $\widehat{A}$ respectively, counting multiplicities.
Now suppose that the eigenvalues $\lambda_i(A)$ of $A$ corresponding to $\lambda$ are from $i = j$ to $i = j + \ell - 1$. Since $A$ is obtained from $\widehat{A}$ by deleting a vertex (i.e., a row and a column corresponding to $\hat{v}$), 
interlacing results (e.g. \cite[Theorem 9.5.1]{godsil_algebraic_graph_theory}) imply that the eigenvalues of $A$ interlace those of $\widehat{A}$, i.e.
\[
\lambda_i(\widehat{A}) \geq \lambda_i(A) \geq \lambda_{i+1}(\widehat{A}) \quad \text{for } i = 1, \ldots, n. 
\]
Now, let $\boldsymbol{w}_1, \ldots, \boldsymbol{w}_n$ be an orthonormal set of eigenvectors of $A$ with eigenvalues $\lambda_1(A), \lambda_2(A),\dots, \lambda_n(A)$ such that $\boldsymbol{w}_j = \boldsymbol{y}_\lambda$ and let $\boldsymbol{u}_1, \ldots, \boldsymbol{u}_{n+1}$ be an orthonormal set of eigenvectors of $\widehat{A}$ with eigenvalues $\lambda_1(\widehat{A}), \lambda_2(\widehat{A}),\dots, \lambda_{n+1}(\widehat{A})$. Let $W_i$ and $U_i$ denote the subspaces spanned by $\boldsymbol{w}_1, \ldots, \boldsymbol{w}_i$ and $\boldsymbol{u}_1, \ldots, \boldsymbol{u}_i$ respectively. Finally, let  
\[
R = \begin{bmatrix}
I_{n \times n} \\
0_{1 \times n}
\end{bmatrix}
\]
be the matrix that deletes from $\widehat{A}$ the row and column corresponding to the vertex $\hat{v}$ via  
\[
R^T \widehat{A} R = A.
\]
Now, for any $i = 1, \dots,n$ we have that $W_i$ is $i$-dimensional and $R^T U_{i-1}$ is at most $(i-1)$-dimensional. Therefore, taking $i = j$, there exists a vector $\boldsymbol{y} \in W_j \cap (R^T U_{j-1})^{\perp}$. This vector satisfies  
\[
\boldsymbol{y}^T R^T \boldsymbol{u}_k = 0 \quad \text{for } k = 1, \ldots, j-1
\]
and so $R \boldsymbol{y} \in U_{j-1}^\perp$. By Rayleigh's inequalities this implies
\[
\lambda_j(\widehat{A}) \geq \frac{(R \boldsymbol{y})^T \widehat{A} (R \boldsymbol{y})}{(R \boldsymbol{y})^T R \boldsymbol{y}} = \frac{\boldsymbol{y}^T A \boldsymbol{y}}{\boldsymbol{y}^T \boldsymbol{y}} \geq \lambda_j(A) = \lambda
\]
with equality if and only if $\boldsymbol{y}$ and $R \boldsymbol{y}$ are eigenvectors of $A$ and $\widehat{A}$, respectively, with eigenvalue $\lambda$. Now if $\boldsymbol{y}$ and $R \boldsymbol{y}$ are eigenvectors with eigenvalue $\lambda$, then necessarily $\boldsymbol{y} = \boldsymbol{y}_\lambda$ because there is only one linearly independent eigenvector with eigenvalue $\lambda$ in $W_j$. But this means $R \boldsymbol{y}_\lambda$ is an eigenvector of $\widehat{A}$ with eigenvalue $\lambda$. However,
\[
(\widehat{A} R \boldsymbol{y}_\lambda)_{\hat{v}} = (R \boldsymbol{y}_\lambda)_{v_m} = ( \boldsymbol{y}_\lambda)_{v_m} \neq 0 \quad \text{and} \quad (R \boldsymbol{y}_\lambda)_{\hat{v}} = 0
\]
where the equality on the left follows from the fact that $\hat{v}$ has only one neighbour, $v_m$, and that $(\boldsymbol{y}_\lambda)_{v_m} \neq 0$ by assumption, and the equality on the right follows from the definition of $R$. Thus,
\[
0 \neq (\widehat{A} R \boldsymbol{y}_\lambda)_{\hat{v}} = \lambda (R \boldsymbol{y}_\lambda)_{\hat{v}} = \lambda \cdot 0 = 0
\]
which is a contradiction. So the Rayleigh inequalities that we obtained are strict and therefore we must have $\lambda_j(\widehat{A}) > \lambda_j(A)$. Using the same argument on $-\widehat{A}$ and $-A$, we get  
\[
 -\lambda_{j+\ell}(\widehat{A}) = \lambda_{n + 2 -j-\ell}(-\widehat{A}) > \lambda_{n + 2 -j-\ell}(-A) = - \lambda_{j+\ell - 1}(A).
\]
and so
\[
\lambda_{j+\ell - 1}(A) > \lambda_{j+\ell}(\widehat{A}).
\]
Since  
\[
\lambda_j(A) = \lambda_{j+1}(A) = \cdots = \lambda_{j+\ell - 1}(A),
\]
the interlacing implies that only the eigenvalues $\lambda_{j+1}(\widehat{A}), \ldots, \lambda_{j+\ell - 1}(\widehat{A})$ of $\widehat{A}$ are equal to $\lambda$. Thus $\lambda$ is an eigenvalue of multiplicity $\mu_{\widehat{A}}(\lambda) = \ell - 1$ of $\widehat{A}$.
\end{proof}

\end{document}